\documentclass{eajam}
\usepackage{charter}
\usepackage[charter]{mathdesign}

\begin{document}

\markboth{T.~Du, Z.~Huang and Y.~Li}{DeepONets' Approximation and Generalization for Learning Operators in SPP}
\title{Approximation and Generalization of DeepONets for Learning  Operators Arising from a Class of Singularly Perturbed Problems}


\author[T.~Du, Z.~Huang and Y.~Li]{Ting Du\affil{1}, Zhongyi Huang\affil{1} and Ye Li\affil{2}\comma\corrauth}
\address{\affilnum{1}\ Tsinghua University, Beijing, China.\\
\affilnum{2}\ Nanjing University of
Aeronautics and Astronautics, Nanjing, China.}
%
%
\emails{{\tt dt20@mails.tsinghua.edu.cn} (T. Du), {\tt zhongyih@
mail.tsinghua.edu.cn} (Z. Huang), {\tt yeli20@nuaa.edu.cn} (Y. Li)}
%
\begin{abstract}
Singularly perturbed problems present inherent difficulty due to the presence of a thin boundary layer in its solution. To overcome this difficulty, we propose using deep operator networks (DeepONets), a method previously shown to be effective in approximating nonlinear operators between infinite-dimensional Banach spaces. In this paper, we demonstrate for the first time the application of DeepONets to one-dimensional singularly perturbed problems, achieving promising results that suggest their potential as a robust tool for solving this class of problems. We consider the convergence rate of the approximation error incurred by the operator networks in approximating the solution operator, and examine the generalization gap and empirical risk, all of which are shown to converge uniformly with respect to the perturbation parameter. By utilizing Shishkin mesh points as locations of the loss function, we conduct several numerical experiments that provide further support for the effectiveness of operator networks in capturing the singular boundary layer behavior. 
\end{abstract}



\maketitle


\section{Introduction}
Singularly perturbed problems (SPP) are widely employed in diverse areas of applied mathematics, such as fluid dynamics, aerodynamics, meteorology, and modeling of semiconductor devices, among others. These problems incorporate a small parameter $\varepsilon$ which typically appears before the highest order derivative term and reflects medium properties. As $\varepsilon\to0$, the derivative (or higher-order derivative) of the solution to the problem becomes infinite, leading to the emergence of boundary or inner layers. In these regions, the solution (or its derivative) undergoes significant changes, while away from the layers, the behavior of the solution is regular and slow. Due to the presence of thin boundary layers, conventional methods for solving this class of problems exhibit either high computational complexity or significant computational expense. For a comprehensive analysis and numerical treatment of singularly perturbed problems, refer to the books by Miller et al. \cite{miller} and Roos et al. \cite{roos_robust_method_spp}. 

Deep neural networks possess the universal approximation property, which allows them to approximate any continuous or measurable finite-dimensional function with arbitrary accuracy. Consequently, they have become a popular tool for solving partial differential equations (PDEs) by serving as a trial space. PDE solvers based on neural networks include Feynman-Kac formula-based methods \cite{beck2018solving, BSDE2, BSDE} and physics-informed neural networks (PINNs) \cite{pinn}. Moreover, machine learning techniques can be used for operator learning (i.e., mapping one infinite-dimensional
Banach space to another) to approximate the underlying solution operator of an equation. To accomplish this, new network structures have been proposed, such as deep operator networks (DeepONets) \cite{lu_deeponet} and neural operators \cite{neural_operator, fno}.

The broad adoption of machine learning techniques in solving PDEs has introduced novel approaches for handling singularly perturbed problems. However, the spectral bias phenomenon \cite{spectral_bias}, which refers to the difficulty of neural networks in capturing sharp fluctuations of the function, poses a significant challenge for directly applying machine learning techniques to solve such problems. In \cite{BLpinn}, it was observed that the traditional PINNs approach did not produce a satisfactory solution nor capture the singular behavior of the boundary layer. To overcome this problem, a deformation of the traditional PINN based on singular perturbation theory is presented in \cite{BLpinn}.

Operator learning techniques have received little attention in the context of singularly perturbed problems, leaving an important research gap to be filled. Our focus is on DeepOnet, a simple yet powerful model. It originated from the universal approximation theorem proved by Chen et al. \cite{chen_universal}, where the shallow network was later extended to a deep network by Lu et al. \cite{lu_deeponet}, resulting in the proposal of DeepONets. This theorem was further extended to more general operators in \cite{deeponet_error}, thereby providing theoretical assurance of the validity of DeepONet to approximate the solution operator. Specifically, for any error bound $\epsilon$, there exists an operator network that can approximate the operator with an accuracy not exceeding $\epsilon$, while also guaranteeing that the size of the network is at most polynomially increasing concerning $1/\epsilon$. Theoretical and experimental results \cite{net_ad, deeponet_error} support the polynomially increasing property of DeepONets in some cases. DeepONet has demonstrated promising performance in various applications, including fractional derivative operators \cite{lu_deeponet}, stochastic differential equations \cite{lu_deeponet}, and advection-diffusion equations \cite{net_ad}. To extend the applicability of DeepONet to more general situations, several deformations have been proposed, including physics-informed DeepONet \cite{pi_deeponet}, multiple-input DeepONet \cite{mionet}, pre-trained DeepONet for multi-physics systems \cite{mmnet2, mmnet1}, proper orthogonal decomposition-based DeepONet (POD-DeepONet) \cite{pod_net}, and Bayesian DeepONet \cite{Bayesian_DeepONet}.

In this manuscript, we employ DeepONets to approximate the solution operator of the convection-diffusion problem, with particular attention to the case of Dirichlet boundary conditions.  As it is known that for non-homogeneous Dirichlet boundary conditions, the solution $u$ can be obtained by subtracting a linear function satisfying the original boundary condition to obtain a solution $u^*$ satisfying the homogeneous boundary condition, we consider the following one-dimensional singularly perturbed problem:
\begin{equation}
    \left\{ \begin{array}{l}
-\varepsilon u''+pu'+qu=f,\quad x\in(0,1),\\
u(0)=u(1)=0,
\end{array} \right.\label{equ}
\end{equation}
where $p(x)\ge \alpha>0$, $q(x)$ are assumed to be sufficiently differentiable. Define the operator $\mathscr{G}:L^2([0,1])\to L^2([0,1]),\ f\mapsto u$, where $u\in H_0^1$ satisfies
\begin{equation}
\varepsilon\int_0^1u'v'dx+\int_0^1pu'vdx+\int_0^1quvdx=\int_0^1fvdx,\quad \forall\ v\in H_0^1.\label{equal_equ}
\end{equation}

Suppose $\mathscr{G}$ is Lipschitz continuous. We utilize DeepONets to approximate $\mathscr{G}$, which is subject to three sources of error: approximation error, generalization error, and optimization error. We focus our attention on the first two sources of error. For singularly perturbed problems, it is critical to determine whether the error is $\varepsilon$-uniform. We adopt a general framework presented in \cite{deeponet_error} to estimate the approximation error and analyze the convergence rate of this error. Our analysis shows that the approximation error is $\varepsilon$-uniform. Given the finite nature of the dataset, training a DeepONet is constrained by limited data availability. Hence, the selection of data assumes critical importance in determining the performance of the model. Our numerical experiments demonstrate that the performance of the trained model is significantly affected by the choice of locations in the loss function. Additionally, the generalization ability of the model, i.e., its capacity to perform well on data points beyond the training dataset, is of particular interest in our study. We select the Shishkin mesh points \cite{shishkin_mesh} as locations and analyze the generalization gap and empirical risk of our model. The results show that the model exhibits strong generalization capability when applied to boundary layer problems, and we present numerical examples to complement our theoretical findings.

The structure of the paper is organized as follows: In Section \ref{sec_preliminary}, we provide a brief introduction to the DeepONets structure and properties of the operator $\mathscr{G}$. Section \ref{sec_approximation} presents an analysis of the approximation error that arises when the operator network approximates the operator $\mathscr{G}$. In Section \ref{sec_generailzation}, we investigate the generalization gap and empirical risk. The Appendix (Section \ref{sec_appendix}) includes some of the proof procedures. In Section \ref{sec_experiemnt}, we provide numerical demonstrations of the efficacy of the proposed method, as well as a comparative analysis of the model's performance under different location selections. Finally, we summarize the work of this study and also discuss potential future work.

\section{Preliminaries\label{sec_preliminary}}
\subsection{DeepONets}

A neural network is a composition of affine functions $\mathcal{A}_l:\ \mathbb{R}^{d_{l-1}}\to\mathbb{R}^{d_l}$ and activation functions $\sigma:\ \ \mathbb{R}\to\mathbb{R}$. Given an input $\boldsymbol{x}\in\mathbb{R}^{d_0}$, the output of the $l_{th}$ layer neural network is defined as $\boldsymbol{y}_l(\boldsymbol{x})=\sigma\circ\mathcal{A}_l(\boldsymbol{y}_{l-1})=\sigma(\boldsymbol{W}_l \boldsymbol{y_{l-1}}(\boldsymbol{x})+\boldsymbol{b}_l)$, where $\boldsymbol{W}_l$ and $\boldsymbol{b}_l$ are the weight matrix and bias vector of the affine function. The output of the last layer of the neural network is the output of the whole neural network and can be expressed as follows:
\begin{equation}
    \mathscr{N}_{\theta}(\boldsymbol{x})=\boldsymbol{y}_L(\boldsymbol{x})=\mathcal{A}_L\circ\sigma\circ\mathcal{A}_{L-1}\dots\dots\circ\sigma\circ\mathcal{A}_2\circ\sigma\circ\mathcal{A}_1(\boldsymbol{x}). \label{nn}
\end{equation}
where $\theta=\{\boldsymbol{W}_l, \boldsymbol{b}_l\}_{1\le l\le L}$ is the set of learnable parameters obtained by minimizing the loss function during the training of the neural network.
The size of a neural network is commonly quantified by its depth $L$ and width $\mathop {\max }\limits_l d_l$. 

In deep neural networks, the ReLU, Sigmoid, and Tanh activation functions are commonly used. Other activation functions, such as Sine \cite{sitzmann2020implicit}, SReLU \cite{liu2020multi}, and adaptive activation functions \cite{jagtap2020adaptive, qian2018adaptive}, have also been proposed.

A DeepONet comprises two sub-networks in the form of \eqref{nn}: one to capture the input function's information at a fixed number of points 
$x_i$, with $i=1,\dots,m $ (i.e., the branch net), and another to encode the output functions' locations (i.e., the trunk net). In this work, we adopt the form presented in \cite{deeponet_error} for the purpose of theoretical analysis, where DeepONet is decomposed into three operators:
\begin{itemize}
    \item Encoder: For a given constant $m$ and $m$-dimensional vector $\{x_i\}_{i=1}^m$, $x_i\in D$, Encoder $\mathscr{E}$ is defined as: 
    \begin{equation*}
        \mathscr{E}:C(D)\to\mathbb{R}^m,\quad f\mapsto \{f(x_i)\}_{i=1}^m.
    \end{equation*}
    \item Approximator: Approximator $\mathscr{A}$ is a neural network, defined by Eq.~\eqref{nn}, takes a specific form given by
    \begin{equation*}
        \mathscr{A}:\mathbb{R}^m\to\mathbb{R}^p,\quad\{f_j\}_{j=1}^m\mapsto\{\mathscr{A}_k\}_{k=1}^p.
    \end{equation*}
    \item Reconstructor: Trunk net $\boldsymbol{\tau}$ is a neural network of the form \eqref{nn} with $d_0=1$ and $d_L=p+1$, which can be expressed as $\boldsymbol{\tau}(x)=(\tau_0(x),\tau_1(x),\dots,\tau_p(x))$, and then the Reconstructor $\mathscr{R}$ induced by $\boldsymbol{\tau}$ can be defined as 
    \begin{equation*}
        \mathscr{R}=\mathscr{R}_{\boldsymbol{\tau}}:\mathbb{R}^p\to C(U),\quad\{\alpha_k\}_{k=1}^p\mapsto\tau_0(\cdot)+\sum\limits_{k=1}^p\alpha_k\tau_k(\cdot).
    \end{equation*}
\end{itemize}
Then the DeepONet, as proposed in \cite{lu_deeponet}, can be expressed as:
\begin{equation}
\mathscr{N}:C(D)\to C(U),\quad f\mapsto\mathscr{R}\circ\mathscr{A}\circ\mathscr{E}(f)(\cdot). \label{deeponet}
\end{equation}
\begin{figure}
    \centering
    \includegraphics[width=0.7\textwidth]{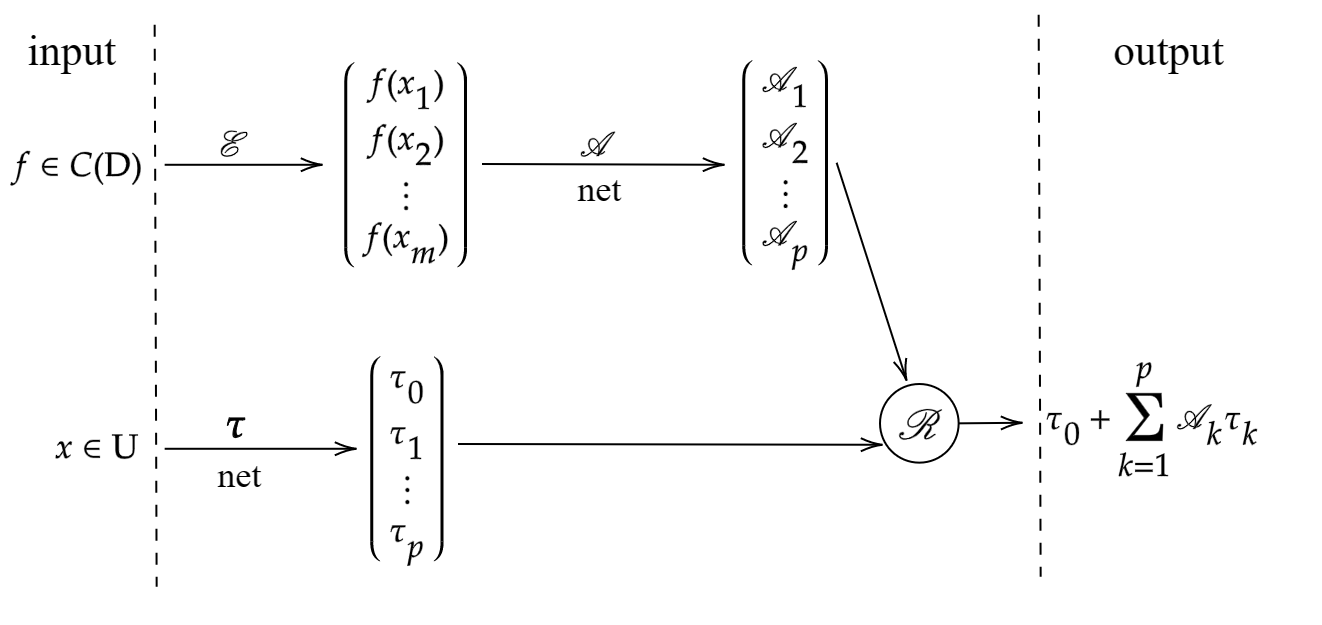}
    \caption{Architecture of DeepONets.}
    \label{architecture_deeponet}
\end{figure}
A diagrammatic representation of the model's structure is shown in Fig.~\ref{architecture_deeponet}.

DeepONet is a neural network model with two inputs: a function $f$ and a location $x$. Given that neural networks require finite-dimensional inputs, $f$ is evaluated at $m$-dimensional location vectors through point-wise evaluation $\mathscr{E}$, resulting in an $m$-dimensional vector that serves as input for the branch net. As a supervised learning model, DeepONet relies on a large dataset of input-output pairs to train the network, a process that is performed offline. Once the network is trained, it can be used as an operator, providing an approximation $\mathscr{N}(f)(x)$ of the equation's solution for any reasonable input function $f$ and location $x$.

\subsection{Prior Estimate}
We introduce two lemmas that describe the fundamental properties of this operator. Lemma \ref{lem_lip} provides a sufficient condition for the Lipschitz continuity of $\mathscr{G}$.
\begin{lemma}
[cf.~Kellogg~et~al.~\cite{kellogg_ssp}] Given $f\in C^{\infty}$ such that 
\begin{equation*}
    |f(x,\varepsilon)|\le K(1+\varepsilon^{-1} \exp(-\alpha(1-x)/\varepsilon)),
\end{equation*}
the solution $\mathscr{G}(f)$ to Eq.~\eqref{equal_equ} satisfies 
\begin{equation*}
    |\mathscr{G}(f)(x)|\le C,
\end{equation*}
where $C$ is a constant that depends only on $K$.
    \label{lem_bounded}
\end{lemma}
 \begin{lemma}
  Under the conditions $0<\varepsilon\ll 1$, $ p(x),\ q(x)\in C^{\infty}$, $p(x)\ge \alpha>0$, $q(x)-\frac{p'(x)}{2}\ge\beta>0$, and $f\in L^2([0,1])$, the weak solution $u$ to Eq.~\eqref{equal_equ} satisfies the inequality 
  \begin{equation*}
      \int_0^1 u^2dx\le C\int_0^1 f^2dx,
  \end{equation*}
  where $C$ is a constant independent of $\varepsilon$.
\label{lem_lip}
\end{lemma}
\begin{proof}
    To prove the inequality stated in the theorem, we start by setting $v=u$ in the Eq.~\eqref{equal_equ}, which gives us
    \begin{equation*}
        \varepsilon\int_0^1(u')^2dx+\int_0^1\left(q(x)-\frac{p'(x)}{2}\right)u^2dx=\int_0^1fudx.
    \end{equation*}
We can then use the condition $q(x)-\frac{p'(x)}{2}\ge\beta>0$, along with the Cauchy-Schwarz inequality, to obtain the desired inequality.
\qed
\end{proof}

\section{Approximation Error\label{sec_approximation}}
In this section, the approximation error generated by the operator network in approximating the solution operator $\mathscr{G}$ is analyzed quantitatively. Given $\mu$ as a probability measure on $L^2([0,1])$, referring to the method in \cite{deeponet_error}, we use $L^2(\mu)$ norm to measure the error between the operator $\mathscr{N}$ and the operator $\mathscr{G}$, i.e.,
\begin{equation}
    \widehat{\mathscr{E}}=\left(\int_{L^2([0,1])}\|\mathscr{G}(f)-\mathscr{N}(f)\|^2_{L^2([0,1])}d\mu(f)\right)^{1/2}.\label{error}
\end{equation}
\begin{remark}
   Although $\mathscr{E}$ is not well-defined on $L^2(D)$, a Borel measurable extension $\Bar{\mathscr{E}}:L^2(D)\to \mathbb{R}^m$ can always be found such that $\mathscr{E}(f)=\Bar{\mathscr{E}}(f)$ for any $f\in C(D)\cap L^2(D)$. By defining the operator network $\bar{\mathscr{N}}=\mathscr{R}\circ\mathscr{A}\circ\bar{\mathscr{E}}:L^2([0,1])\to L^2([0,1])$, the error formula \eqref{error} can be well-defined. The existence of the operator network $\mathscr{N}$ is guaranteed by the universal approximation theorem, specifically by the version proposed by Lanthaler et al. \cite{deeponet_error}.
\end{remark}

The universal approximation theorem of Chen and Chen \cite{chen_universal} and its generalized version \cite{deeponet_error} both guarantee the existence of a DeepONet of the form \eqref{deeponet} that can approximate a given nonlinear operator $\mathscr{G}$ with an error not exceeding a pre-specified bound $\epsilon$. However, these theorems do not provide specific information on determining the configuration of a DeepONet, such as the size of the branch and trunk networks. Our goal is to better understand the performance of DeepONets in approximating nonlinear operators by conducting a quantitative analysis of their efficiency, specifically investigating how variations in network size and structure affect their approximation accuracy and computational efficiency. Lanthaler et al. \cite{deeponet_error} propose a general framework for error estimation, which includes results for general operators and several specific examples. Although operator networks have not yet been applied to singularly perturbed problems, the universal approximation theorem assures the existence of a DeepONet of the form \eqref{deeponet} that can approximate the solution operators for this class of problems. Our main focus in this section is to quantitatively analyze the efficiency of DeepONets in approximating these operators.

\subsection{Error Estimate}
To estimate \eqref{error}, it is necessary to first identify the underlying space. Let $\mu=\mathcal{N}(\mathbb{E}_{\mu},\Gamma)$ be the Gaussian measure on the Hilbert space $X$, where $\mathbb{E}_{\mu}=\int_{X}ud\mu(u)$ is the mean and $\Gamma$ is the covariance operator satisfying that, for any $k,l\in X$ and $\boldsymbol{x}\sim\mu$, 
\begin{equation*}
    \langle k,\Gamma l\rangle=\mathbb{E}\langle k,\boldsymbol{x}-\mathbb{E}_{\mu}\rangle\langle\boldsymbol{x}-\mathbb{E}_{\mu},l\rangle,
\end{equation*}
which implies that $\Gamma=\mathbb{E}(\boldsymbol{x}-\mathbb{E}_{\mu})\otimes(\boldsymbol{x}-\mathbb{E}_{\mu})$.

Let $\{\lambda_k,\phi_k\}_{k=1}^{+\infty}$ denote eigenvalue-eigenvector pairs of the operator $\Gamma$, where
\begin{equation*}
    \lambda_1\ge\lambda_2\ge\dots,
\end{equation*}
and the set $\{\phi_k\}_{k=1}^{+\infty}$ is orthogonal. Additionally, let $\{\xi_k\}_{k=1}^{+\infty}$ be a set of independent identically distributed (i.i.d.) random variables with $\xi_1\sim\mathcal{N}(0,1)$. We can express $\boldsymbol{x}\sim\mu$ using the $Karhunen-Lo\grave{e}ve\ expansion$ \cite{loeve2017probability} as
\begin{equation*}
\boldsymbol{x}=m+\sum\limits_{k=1}^{+\infty}\sqrt{\lambda_k}\xi_k\phi_k. 
\end{equation*}
This series converges in the quadratic mean (q.m.) uniformly on $[0,1]$.

In this work, we consider the input space to be endowed with a Gaussian measure $\mu=\mathcal{N}(0,\mathcal{C})$, where $\mathcal{C}$ is a covariance operator represented as an integral operator with a kernel, as previously employed in \cite{deeponet_error,lu_deeponet}. Specifically, the kernel $k_p(x_1,x_2)$ is the periodization of the radial basis function (RBF) kernel $k_l(x_1,x_2)$, where $k_l(x_1,x_2)=\exp(-|x_1-x_2|^2/2l^2)$ is a commonly used kernel. The periodization is defined as 
\begin{equation*}
    k_p(x_1,x_2)=\mathop \sum \nolimits_{n\in\mathbb{Z}}\exp{(-|x_1-x_2-n|^2/2l^2}).
\end{equation*}
The parameter $l>0$ in the RBF kernel controls the smoothness of the sampled function, with larger values of $l$ corresponding to smoother functions. The eigenvalues and eigenvectors of the covariance operator $\mathcal{C}$ are characterized by
\begin{equation*}
    \lambda_k=\sqrt{2\pi}le^{-2\pi^2k^2l^2},\ \phi_k(x)=e^{i2\pi kx},
\end{equation*}
respectively, where $k\in\mathbb{Z}$. This can be confirmed by noting that
\begin{equation*}
    \int_0^1 k_p(x,y)\phi_k(y)dy=\lambda_k\phi_k(x).
\end{equation*}
Then by the K-L expansion, any $f\sim\mu$ has the representation as
\begin{equation}
   f=\sum\limits_{k=-\infty}^{+\infty}\sqrt{\lambda_k}\xi_k\phi_k,\label{gaussian process} 
\end{equation}
where $\{\xi_k\}_{k=-\infty}^{+\infty}$ is an i.i.d. sequence of $\mathcal{N}(0, 1)$ random variables.

Subsequently, regarding the approximation error \eqref{error}, we arrive at the following result, whose rigorous proof is presented in the subsequent subsection.

\begin{theorem}
Assuming the solution operator for Eq.~\eqref{equ} is Lipschitz continuous, 
for any fixed, arbitrarily small $\gamma\in(0,1)$ and $\eta>0$ (which may depend on $p$), there exists a ReLU Trunk net $\boldsymbol{\tau}$ with $width(\boldsymbol{\tau})\lesssim\max\{p,(\eta/p)^{\frac{1}{\gamma-2}}\}$ and $depth(\boldsymbol{\tau})\le1$, and a ReLU Branch net $\boldsymbol{\beta}$ with $\text{width}(\boldsymbol{\beta})\lesssim \max\{m,p\}$ and $\text{depth}(\boldsymbol{\beta})\le1$. These nets yield an operator network $\mathscr{N}$ such that the following estimate holds: 
\begin{equation*}
    \widehat{\mathscr{E}}(\mathscr{N})=\|\mathscr{G}-\mathscr{N}\|_{L^2(\mu)}\lesssim \eta+e^{-\pi^2l^2K^2}+e^{-\pi^2l^2M^2},
\end{equation*}
where $K=(p-1)/2,\ M=(m-1)/2$.
\label{thm_approximation_error}
\end{theorem}

 Theorem \ref{thm_approximation_error} shows that different choices of $\eta$ correspond to different sizes of the trunk network and various error bounds for $\widehat{\mathscr{E}}$. Notably, by setting $\eta=e^{-\pi^2l^2K^2}$, the approximation error $\widehat{\mathscr{E}}$ exhibits super-exponential decay with respect to $p$. However, this comes at the cost of an exponential growth in the size of the trunk net, which is highly undesirable. A preferable scenario is for the size of the trunk net to exhibit polynomial growth with respect to $p$, which can be achieved by setting $\eta=p^{-n}$, where $n$ is a fixed, positive constant. And for notational simplicity, we adopt $\eta=p^{-n}$ throughout the remainder of this manuscript. In other cases, identical reasoning leads to the corresponding conclusions. By applying Theorem \ref{thm_approximation_error} and introducing the notation $\text{size}(\Phi)$ to represent the number of neurons in neural network $\Phi$, we arrive at the following theorem in a straightforward manner.

\begin{theorem}
Under the assumption of Lipschitz continuity of the solution operator for Eq.~\eqref{equal_equ}, a ReLU Branch net $\boldsymbol{\beta}$ of $\text{size}(\boldsymbol{\beta})\lesssim \max\{m,p\}$ and a ReLU Trunk net $\boldsymbol{\tau}$ of $\text{size}(\boldsymbol{\tau})\lesssim p^{\frac{n+1}{2-\gamma}}$ can be constructed, for any positive number $n$ and any fixed, arbitrarily small constant $\gamma\in(0,1)$. These nets yield an operator network $\mathscr{N}$ that satisfies the following estimate:
\begin{equation*}
    \widehat{\mathscr{E}}(\mathscr{N})=\|\mathscr{G}-\mathscr{N}\|_{L^2(\mu)}\lesssim p^{-n}+e^{-\pi^2l^2K^2}+e^{-\pi^2l^2M^2}.
\end{equation*}
\label{cor_approximation error}
\end{theorem}

\subsection{Proof of Theorem \ref{thm_approximation_error}}
Following the workflow outlined in \cite{deeponet_error}, we introduce here the operators $\mathscr{D}$ and $\mathscr{P}$ as approximations to the inverses of $\mathscr{E}$ and $\mathscr{R}$, and employ $\widehat{\mathscr{E}}_{\mathscr{E}}$, $\widehat{\mathscr{E}}_{\mathscr{R}}$, and $\widehat{\mathscr{E}}_{\mathscr{A}}$ to quantify the differences between $\mathscr{D} \circ \mathscr{E}$ and the identity operator $Id: L^2([0,1]) \to L^2([0,1])$, between $\mathscr{R} \circ \mathscr{P}$ and $Id: L^2([0,1]) \to L^2([0,1])$, and between $\mathscr{A}$ and $\mathscr{P} \circ \mathscr{G} \circ \mathscr{D}$, respectively. The formulae for these errors are shown below:
\begin{equation*}
\widehat{\mathscr{E}}_{\mathscr{E}}=\left(\int_{L^2([0,1])}\|\mathscr{D}\circ\mathscr{E}(f)-f\|_{L^2([0,1])}^2d\mu(f)\right)^{1/2},
\end{equation*}
\begin{equation*}
\widehat{\mathscr{E}}_{\mathscr{R}}=\left(\int_{L^2([0,1])}\|\mathscr{R}\circ\mathscr{P}(u)-u\|_{L^2([0,1])}^2d(\mathscr{G}_{\#}\mu)(u)\right)^{1/2},
\end{equation*}
\begin{equation*}
\widehat{\mathscr{E}}_{\mathscr{A}}=\left(\int_{\mathbb{R}^m}\|\mathscr{A}(\boldsymbol{y})-\mathscr{P}\circ\mathscr{G}\circ\mathscr{D}(\boldsymbol{y})\|_{l^2(\mathscr{R}^p)}^2d(\mathscr{E}_{\#}\mu)(\boldsymbol{y})\right)^{1/2}.
\end{equation*}
Then, we can decompose the approximation error \eqref{error} into the following terms:
\begin{lemma}
[cf.~Lanthaler~et~al.~\cite{deeponet_error}]
The approximation error \eqref{error} associated with $\mathscr{G}$ and the operator network $\mathscr{N}=\mathscr{R}\circ\mathscr{A}\circ\mathscr{E}$ can be bounded as follows:
\begin{equation*}
\widehat{\mathscr{E}}\le Lip(\mathscr{G})Lip(\mathscr{R}\circ\mathscr{P})\widehat{\mathscr{E}}_{\mathscr{E}}+Lip(\mathscr{R})\widehat{\mathscr{E}}_{\mathscr{A}}+\widehat{\mathscr{E}}_{\mathscr{R}},
\end{equation*}
where $Lip(\cdot)$ denotes the Lipschitz norm.
\label{error_decomposition}
\end{lemma}

According to Lemma \ref{error_decomposition}, the approximation error $\widehat{\mathscr{E}}$ can be effectively constrained by controlling $\widehat{\mathscr{E}}_{\mathscr{E}}$, $\widehat{\mathscr{E}}_{\mathscr{A}}$, and $\widehat{\mathscr{E}}_{\mathscr{R}}$. We begin by presenting the result for the error term $\widehat{\mathscr{E}}_{\mathscr{E}}$.
\begin{lemma}
    Let $x_j=j/m,\ j=1,2,3,\dots,m$ with $m = 2M+1$ be $m$ equidistant points on the interval $[0,1]$. Consider the encoder $\mathscr{E}:L^2([0,1])\to\mathbb{R}^m$, defined as $\mathscr{E}(f)=(f(x _1),f(x_2),\dots,f(x_m))$, then, there exists an operator $\mathscr{D}:\mathbb{R}^m\to L^2([0,1])$ such that 
    \begin{equation*}
  \widehat{\mathscr{E}}_{\mathscr{E}}=\|\mathscr{D}\circ\mathscr{E}-Id\|_{L^2(\mu)}\lesssim e^{-\pi^2l^2M^2}.
    \end{equation*}
\label{encoding_error}
\end{lemma}
\begin{remark}
Here the operator $\mathscr{D}$ is given by: 
\begin{equation*}
    \mathscr{D}(f_1,f_2,\dots,f_m)(x)=\sum\limits_{k=-M}^{M}\widehat{f}_ke^{i2\pi kx},
\end{equation*}
where 
\begin{equation*}
    \widehat{f}_ k=\frac{1}{m}\sum\limits_{j=1}^m f_je^{-\frac{i2\pi jk}{m}}.
\end{equation*}
The proof of this result is omitted for brevity, but for a detailed derivation, we refer the reader to the proof of Lemma 3.8 in \cite{deeponet_error}. The use of the symbol \textquotedbl$\lesssim$\textquotedbl{} in the conclusion of Lemma \ref{encoding_error} implies that a constant $C$, which is independent of parameters such as $m$, $p$, $\varepsilon$ and $l$, has been omitted from the right-hand side of the inequality. We will adopt this notation without further explanation in subsequent sections.
\label{rek_encoder}
\end{remark}

We now turn to the analysis of the remaining error terms, namely $\widehat{\mathscr{E}}_{\mathscr{R}}$ and $\widehat{\mathscr{E}}_{\mathscr{A}}$. To proceed, we assume that $p$ is an odd integer, specifically $p=2K+1$. For simplicity, we denote $\mathscr{G}_{\#}\mu$ as $\nu$. We introduce the operator $\widetilde{\mathscr{R}}$, defined as
\begin{equation*}
   \widetilde{\mathscr{R}}(a_{-K},a_{-K+1},\dots,a_K)=\mathbb{E}_{\nu}+\sum\limits_{j=-K}^{K}a_j\widetilde{\tau}_j, 
\end{equation*}
where $\{\widetilde{\tau}_j\}_{j=-K}^K$ are obtained from $\{\mathscr{G}(\phi_j)\}_{j=-K}^K$ (i.e., $\{\mathscr{G}(e^{i2\pi jx})\}_{j=-K}^K$) by Gram-Schmidt orthogonalization. Let $\mathscr{P}$ be the operator defined by
\begin{equation*}
    \mathscr{P}(u)=((u-\mathbb{E}_{\nu},\widetilde{\tau}_{-K}),\dots,(u-\mathbb{E}_{\nu},\widetilde{\tau}_{K})),
\end{equation*}
then the $L^2(\nu)$ norm of the difference between $\widetilde{\mathscr{R}}\circ\mathscr{P}$ and $Id$ can be estimated in the following way:
\begin{align}
\int_{L^2}\|\widetilde{\mathscr{R}}\circ\mathscr{P}(u)-u\|_{L^2}^2d\nu(u)&=\int_{L^2}\mathop {\inf }\limits_{ \widehat{u}\in\ \mathbb{E}_{\nu}+span\{\widetilde{\tau}_{-K},\dots,\widetilde{\tau}_{K}\}}\|\widehat{u}-u\|_{L^2}^2d\nu(u)\nonumber\\
&=\int_{L^2}\mathop {\inf }\limits_{ \widehat{u}\in\ \mathscr{G}(\mathbb{E}_{\mu})+span\{\mathscr{G}(\phi_{-K}),\dots.\mathscr{G}(\phi_K)\}}\|\widehat{u}-u\|_{L^2}^2d\nu(u)\nonumber\\
&=\int_{L^2}\mathop {\inf }\limits_{ \widehat{f}\in\ \mathbb{E}_{\mu}+span\{\phi_{-K},\dots,\phi_K\}}\|\mathscr{G}(\widehat{f})-\mathscr{G}(f)\|_{L^2}^2d\mu(f)\nonumber\\
&\le Lip(\mathscr{G})^2\int_{L^2}\mathop {\inf }\limits_{ \widehat{f}\in\ \mathbb{E}_{\mu}+span\{\phi_{-K},\dots,\phi_K\}}\|\widehat{f}-f\|_{L^2}^2d\mu(f)\nonumber\\
&=Lip(\mathscr{G})^2\int_{L^2}\|\mathbb{E}_{\mu}+\sum\limits_{j=-K}^K(f-\mathbb{E}_{\mu},\phi_j)\phi_j-f\|_{L^2}^2d\mu(f)\nonumber\\
&=Lip(\mathscr{G})^2\sum\limits_{|j|>K}\lambda_j^{\mu}\nonumber\\
&\le Lip(\mathscr{G})^2e^{-2\pi^2l^2K^2}.\label{error_tilde{R}}
\end{align} 
    
The bases $\{\widetilde{\tau}_j\}_{j=-K}^K$ in $\widetilde{\mathscr{R}}$ are not represented by neural networks. To obtain the desired operator $\mathscr{R}$ based on $\widetilde{\mathscr{R}}$, it is reasonable to approximate both $\mathbb{E}_{\nu}$ and $\{\widetilde{\tau}_j\}_{j=-K}^K$ using a set of neural networks. Specifically, we aim to approximate $\mathscr{G}(\mathbb{E}_{\mu})$ and $\{\mathscr{G}(e^{i2\pi jx})\}_{j=-K}^K$, respectively. As $\mathbb{E}_{\mu}=0$, it follows that $\mathbb{E}_{\nu}=0$, allowing us to set $\tau^0=0$ as the neural network that approximates $\mathbb{E}_{\nu}$. We then demonstrate the existence of a set of neural networks, which can approximate $\{\widetilde{\tau}_j\}_{j=-K}^K$ with an error that does not exceed a given bound. This conclusion is supported by the following lemma. For a similar derivation, one can refer to the proof of Lemma 1 in \cite{refined_mesh_spp}. Therefore, the proof for this case is omitted for brevity.
\begin{lemma}
    Consider the Shishkin mesh, a piecewise uniform grid on $[0,1]$, defined by
    \begin{equation}
       \Omega=\{x_j:x_j=jh,0\le j\le \frac{J}{2};\ x_j=1-\sigma+(j-\frac{J}{2})H,\frac{J}{2}\le j\le J\},\label{mesh} 
    \end{equation}
where $h=2(1-\sigma)/J,\ H=2\sigma/J,\ \sigma=\min\{1/2,\ 2\varepsilon \ln{J}/\alpha\}$.
Considering a sufficiently smooth function $f$, we define the corresponding solution to Eq.~\eqref{equ} as $u$. Then, for the piecewise linear interpolation function $u_L(x)$ on $\{(x_i,u(x_i))\}_{i=0}^J$, we have the error estimate 
\begin{equation*}
    |u(x)-u_L(x)|\lesssim\left\{\begin{array}{ll}
1/J^2, & x\le1-\sigma,\\
   \ln^2J/J^2,  &x> 1-\sigma.
\end{array}\right.
\end{equation*}
\label{lem_shishkin_interpolation}
\end{lemma}

Meanwhile, 
\begin{equation*}
            u_L=\sum\limits_{n=1}^J\left((u(x_n)-u(x_{n-1}))\frac{x-x_n}{x_n-x_{n-1}}+u(x_n)\right)1_{[x_{n-1},x_n]}=\sum\limits_{n=0}^Ju(x_n)\mathscr{L}_n(x),
    \end{equation*}
    $\mathscr{L}_n(x)$ is the piecewise linear nodal basis supported on
the sub-interval $[x_{n-1}, x_{n+1})$, which can be represented by a neural network of depth 1 and width 3 as shown in \cite{relu_fem}, i.e.,
       \begin{equation*}
       \begin{aligned}
                  \mathscr{L}_n(x)=&\frac{1}{x_n-x_{n-1}}{\rm ReLU}(x-x_{n-1})-\left(\frac{1}{x_n-x_{n-1}}+\frac{1}{x_{n+1}-x_n}\right){\rm ReLU}(x-x_n)\\
                  &+\frac{1}{x_{n+1}-x_n}{\rm ReLU}(x-x_{n+1}).
       \end{aligned}
       \end{equation*}
Thus, there exists a neural network $\tau$ of width $\mathcal{O}(J)$ and depth no more than 1 such that
    \begin{equation}
        \|u-\tau\|_{L^{\infty}}\le\|u-u_L\|_{L^{\infty}}+\|u_L-\tau\|_{L^{\infty}}\lesssim\frac{\ln^2J}{J^2}.\label{nn_aproximate_basis}
    \end{equation}
    
The inequality stated in \eqref{nn_aproximate_basis} holds for all $u\in \{\mathscr{G}(e^{i2\pi jx})\}_{j=-K}^K$. By invoking the definition of $\widetilde{\tau}_j$, along with \eqref{nn_aproximate_basis}, it can be inferred that for a given $\delta > 0$ such that $\delta\gtrsim\ln^2J/J^2$, neural networks $\tau^0$ and $\{\tau_j\}_{j=-K}^K$ with a width of $\mathcal{O}(J)$, and depth no more than 1, can be found to satisfy
\begin{equation*}
   \max\{\|\tau^0-\mathbb{E}_{\nu}\|_{L^{\infty}},\mathop {\max }\limits_{j=-K,\dots,K}\|\tau_j-\widetilde{\tau}_j\|_{L^{\infty}}\}<\delta. 
\end{equation*}
We define the reconstructor, induced by $\boldsymbol{\tau}=(\tau^0,\tau_{-K},\tau_{-K+1},\dots,\tau_K)$, as
\begin{equation*}
    \mathscr{R}:\mathscr{R}(a_{-K},a_{-K+1},\dots,a_{K})=\tau^0+\sum\limits_{j=-K}^{K}a_j \tau_j,
\end{equation*}
which can be used to approximate the operator $\widetilde{\mathscr{R}}$. By combining this with the estimate \eqref{error_tilde{R}}, we obtain the following error estimate:
        \begin{equation*}
        \begin{aligned}
        \|\mathscr{R}\circ\mathscr{P}-Id\|_{L^2(\nu)}\le&\|\widetilde{\mathscr{R}}\circ\mathscr{P}-Id\|_{L^2(\nu)}+\|\mathscr{R}\circ\mathscr{P}-\widetilde{\mathscr{R}}\circ\mathscr{P}\|_{L^2(\nu)}\\
       \le& Lip(\mathscr{G})e^{-\pi^2l^2K^2}+\left(\int_{L^2}\|\sum\limits_{j=-K}^{K}(u,\widetilde{\tau}_j)(\tau_j-\widetilde{\tau}_j)\|_{L^2}^2d\nu(u)\right)^{1/2}\\
        \le&Lip(\mathscr{G})e^{-\pi^2l^2K^2}+ pLip(\mathscr{G})\left(\sum\limits_{k\in \mathbb{Z}}\lambda_k^{\mu}\right)^{1/2}\mathop {\max }\limits_{j=-K,\dots,K}\|\tau_j-\widetilde{\tau}_j\|_{L^2}\\
        \le&Lip(\mathscr{G})\left(e^{-\pi^2l^2K^2}+\delta p\right).
        \end{aligned}
    \end{equation*}

By assuming the Lipschitz continuity of $\mathscr{G}$, and introducing $\eta=p\delta$, we can summarize the above results in the following lemma.
\begin{lemma}
There exists a set of neural networks $\boldsymbol{\tau}=(\tau^0,\tau_{-K},\tau_{-K+1},\dots,\tau_K)$ in the form of \eqref{nn} with
$width(\boldsymbol{\tau})\lesssim\max\{p,(\eta/p)^{\frac{1}{\gamma-2}}\}$, $depth(\boldsymbol{\tau})\le1$, that satisfies
\begin{equation*}
   \max\{\|\tau^0-\mathbb{E}_{\nu}\|_{L^{\infty}},\mathop {\max }\limits_{j=-K,\dots,K}\|\tau_j-\widetilde{\tau}_j\|_{L^{\infty}}\}<\eta/p, 
\end{equation*}
where $\{\widetilde{\tau}_j\}_{j=-K}^K$ are obtained from $\{\mathscr{G}(e^{i2\pi jx})\}_{j=-K}^K$ by Gram-Schmidt orthogonalization, and $\gamma\in(0,1)$ is a fixed, arbitrarily small constant. This set of neural networks $\boldsymbol{\tau}$ can serve as a trunk net in a DeepONet. The Reconstructor $\mathscr{R}:\mathbb{R}^p\to L^2([0,1])(p=2K+1)$ induced by $\boldsymbol{\tau}$ is defined as
\begin{equation*}
    \mathscr{R}(a_{-K},a_{-K+1},\dots,a_K)=\tau^0+\sum\limits_{j=-K}^Ka_j\tau_j.
\end{equation*}
Then there exists an operator $\mathscr{P}:L^2([0,1])\to\mathbb{R}^p$ such that the error term $\widehat{\mathscr{E}}_{\mathscr{R}}$ satisfies
\begin{equation*}
\widehat{\mathscr{E}}_{\mathscr{R}}=\|\mathscr{R}\circ\mathscr{P}-Id\|_{L^2(\nu)}\lesssim e^{-\pi^2l^2K^2}+\eta.
\end{equation*}
\label{Reconstruction Error}
\end{lemma}

The operator $\mathscr{P}$ (constructed in Lemma \ref{Reconstruction Error}) and the operator $\mathscr{D}$ (constructed in Lemma \ref{encoding_error}) are linear mappings. Hence, $\mathscr{P}\circ\mathscr{G}\circ\mathscr{D}$ is a linear mapping from $\mathbb{R}^m$ to $\mathbb{R}^p$, then there exists a matrix $A$ such that $\mathscr{P}\circ\mathscr{G}\circ\mathscr{D}=A$. The network approximator $\mathscr{A}$, as an approximation of $\mathscr{P}\circ\mathscr{G}\circ\mathscr{D}$, can be taken as $\mathscr{A}:\mathbb{R}^m\to\mathbb{R}^p,\ \boldsymbol{x}\mapsto A\boldsymbol{x}$ or $\boldsymbol{x}\mapsto {\rm ReLU}(A\boldsymbol{x})-{\rm ReLU}(-A\boldsymbol{x})$, that is, with respect to the error term $\widehat{\mathscr{E}}_{\mathscr{A}}$, we can easily obtain the following result:
\begin{lemma}
There exists a neural network $\mathscr{A}$ of $width(\mathscr{A})\lesssim\max\{m,p\}$, $depth(\mathscr{A})\le1,$ such that
\begin{equation*}
    \widehat{\mathscr{E}}_{\mathscr{A}}=    \left(\int_{\mathbb{R}^m}\|\mathscr{A}(\boldsymbol{x})-\mathscr{P}\circ\mathscr{G}\circ\mathscr{D}(\boldsymbol{x})\|^2_{l^2(\mathbb{R}^p)}d(\mathscr{E}_{\#} \mu)(\boldsymbol{x})\right)^{1/2}=0.
\end{equation*}
\label{approxi_error}
\end{lemma}

We note that the Lipschitz norm of $\mathscr{R}\circ\mathscr{P}$ is bounded by $2$. This observation, coupled with the fact that the Branch net in DeepONet corresponds to $\mathscr{A}\circ\mathscr{E}$, allows us to rigorously establish the validity of Theorem \ref{thm_approximation_error}. The proof is completed through a careful combination of several lemmas, namely, Lemma \ref{error_decomposition}, Lemmas \ref{encoding_error} and \ref{Reconstruction Error}, and finally, Lemma \ref{approxi_error}.

\section{Generalization Capability\label{sec_generailzation}}
Acquiring the solution operator $\mathscr{G}$ for partial differential equations (PDEs) directly proves challenging in most cases, rendering the computation of the error term \eqref{error} equally difficult. Therefore, a proxy measure is necessary to quantify the difference between the operator network and operator $\mathscr{G}$. Numerical techniques are thus implemented to discretize the integration and estimate the error \eqref{error}. For example, the Monte Carlo method is commonly used to approximate the integration over $L^2([0,1])$, while both stochastic and deterministic methods, such as the trapezoidal rule, can be utilized for integration over $[0,1]$.

For problems with a solution that has a bounded gradient, locations can be randomly selected from the domain $[0,1]$ or taken as equidistant points. However, for singularly perturbed problems like \eqref{equ} with relatively small parameter $\varepsilon$, the solution exhibits an exponential boundary layer near $x=1$.  Within this region, the solution's gradient is relatively large, and if the aforementioned sampling methods are utilized, a sufficiently large number of points will be required to capture the boundary layer information accurately. Moreover, in these cases, the discrete form error may depend on $1/\varepsilon$, which is undesirable.

Utilizing the right Riemann sum on a Shishkin mesh \cite{shishkin_mesh} allows us to approximate the original integral by partitioning the interval into subintervals with varying widths. Specifically, the Shishkin mesh employs a dense clustering of points near the boundary layer, where the solution changes rapidly, to ensure accurate approximation in this critical region. This technique has been shown to effectively capture the behavior of singularly perturbed problems and has been widely adopted in numerical simulations and modeling.

We consider $N$ samples $F_1,F_2,\dots,F_N$ drawn independently and identically from a common distribution $\mu$, and we let $\{y_j\}_{j=0}^J$ denote the Shishkin mesh points on the interval $[0,1]$, as defined in \eqref{mesh}. We approximate two integrals in the error term \eqref{error} using Monte Carlo integration and numerical integration, respectively. Then the empirical risk $\mathscr{E}_{N,J}$ can be derived, which takes the following form:
\begin{equation}
    \mathscr{E}_{N,J}=\left(\frac{1}{N}\sum\limits_{n=1}^{N}\sum\limits_{j=0}^{J}\omega_j|\mathscr{G}(F_n)(y_j)-\mathscr{N}(F_n)(y_j)| ^2\right)^{1/2},\label{empirical risk}
\end{equation}
where the weights $\omega_i$ are determined as follows: $\omega_0=0$, $\omega_i=h$ for $1\le i\le \frac{J}{2}$, and $\omega_i=H$ for $\frac{J}{2}<i\le J$, with $h$ and $H$ defined in Lemma \ref{lem_shishkin_interpolation}.

\subsection{Main Results}
Ensuring that a machine learning model can generalize well to new, unseen data beyond the training dataset is crucial. In this context, we focus on assessing the generalization ability of our machine learning model, denoted by $\widehat{\mathscr{N}}$, which is learned via a given learning algorithm. To measure the model's generalization performance, we use the concept of the generalization gap (see \cite{generalization}), which quantifies the difference between the approximation error and the empirical risk on $\widehat{\mathscr{N}}$:
\begin{equation}
\mathscr{E}_{gen}=\mathbb{E}\left|\widehat{\mathscr{E}}(\widehat{\mathscr{N}})-\mathscr{E}_{N,J}(\widehat{\mathscr{N}})\right|\triangleq\mathscr{E}_{diff}(\widehat{\mathscr{N}}).\label{gen_error}
\end{equation}
To control this gap, we first introduce the following assumptions.
\begin{assumption}
    There exists $q\in [4,+\infty)$ such that $\mathbb{E}_{f\sim\mu}\|\mathscr{G}(f)-\mathscr{N}(f)\|^q_{L^2}\lesssim 1$.\label{assump_1}
\end{assumption}
\begin{assumption}
There exists an operator $\widetilde{\mathscr{N}}:L^2([0,1])\to C^1([0,1])$ such that for any $x\in [0,1]$,
\begin{equation}
\mathbb{E}_{f\sim\mu}|\mathscr{N}(f)(x)-\widetilde{\mathscr{N}}(f)(x)|^2\lesssim 1/N^{1/2},\label{equ_assump_2}
\end{equation}
and\begin{equation}
    \mathbb{E}_{f\sim\mu}\left|\frac{d}{dx}(\mathscr{G}(f)(x)-\widetilde{\mathscr{N}}(f)(x))^2\right|\lesssim 1+\varepsilon^{-1}e^{-\alpha(1-x)/\varepsilon}.\label{equ_assump_3}
\end{equation}
\label{assump_2}
\end{assumption}
\begin{remark}
Activation functions like ReLU are not differentiable everywhere, which means that the derivative of $\mathscr{N}(f)(x)$ may be undefined at some points, and present challenges for error estimation. To overcome this issue, we introduce the operator $\widetilde{\mathscr{N}}$, which need not be an operator network, and put forth Assumption \ref{assump_2}. The incorporation of Assumption \ref{assump_1} permits the error estimation for Monte Carlo integration, while the inclusion of Assumption \ref{assump_2} enables the error estimation for deterministic numerical integration. The reasonableness of these assumptions can be briefly explained: when $\mathbb{E}_{f\sim\mu}\|f\|_{L^2}\lesssim 1$, it follows that $\mathbb{E}_{f\sim\mu}\|\mathscr{G}(f)\|_{L^2}\lesssim 1$ and for any $x\in [0,1]$,
\begin{equation*}
    \mathbb{E}_{f\sim\mu}\left|\frac{d}{dx}\mathscr{G}(f)^2\right|\lesssim 1+\varepsilon^{-1}e^{-\alpha(1-x)/\varepsilon}.
\end{equation*}
Furthermore, since the operator network $\mathscr{N}$ is an approximation of the target operator $\mathscr{G}$, it is reasonable to make similar assumptions for $\mathscr{G}-\mathscr{N}$. Notably, the operator network constructed in Theorem \ref{thm_approximation_error} and Theorem \ref{cor_approximation error} satisfy these assumptions.
\end{remark}

These assumptions allow us to obtain an estimate for the generalization gap, which will be rigorously proved in the next subsection.
\begin{theorem}
Under Assumptions \ref{assump_1} and \ref{assump_2}, and assuming the Lipschitz continuity of $\mathscr{G}$, a bound for the generalization gap \eqref{gen_error} can be established as follows:
    \begin{equation}
        \mathscr{E}_{gen}=\mathbb{E}\left|\widehat{\mathscr{E}}(\widehat{\mathscr{N}})-\mathscr{E}_{N,J}(\widehat{\mathscr{N}})\right|\lesssim\frac{1}{N^{1/4}}+\frac{\sqrt{\ln{J}}}{J^{1/2}}.\label{equ_generelization gap}
    \end{equation}
\label{generelization gap}
\end{theorem}
\begin{remark}
Both error expressions \eqref{error} and \eqref{empirical risk} involve a square root symbol on the outside. Therefore, the $1/N^{1/4}$ term in Eq.~\eqref{equ_generelization gap} corresponds to the standard Monte Carlo error rate, while $\sqrt{\ln{J}}/J^{1/2}$ corresponds to the error of the numerical integration method. Theorem \ref{generelization gap} tells us that increasing the number of sampling points can reduce the generalization gap.
\end{remark}

As noted previously, the true error of our operator network $\mathscr{N}$ in approximating $\mathscr{G}$ is unavailable. Instead, we rely on the empirical error as a proxy for assessing the quality of the approximation. Therefore, analyzing this empirical error is crucial. We expect the existence of an operator network in the form of \eqref{deeponet}, such that the corresponding empirical risk exhibits convergence properties with various factors that affect the performance of the neural network, including the dimensions $m,\ p$ of the input and output spaces, and the number of sample points $N,\ J$. Importantly, we hope this error is independent of the negative power of the parameter $\varepsilon$; otherwise, if $\varepsilon$ is too small, the corresponding error will become unacceptably large, rendering the operator network unusable for approximating the solution operator of Eq.~\eqref{equ}. To address this issue, we present the following result, with proof provided in the next subsection.

\begin{theorem}
Assuming Lipschitz continuity of operator $\mathscr{G}$ and $m\le p$, we can construct a ReLU branch network $\boldsymbol{\beta}$ and a ReLU trunk network $\boldsymbol{\tau}$, with $\text{size}(\boldsymbol{\beta})\lesssim p$ and $\text{size}(\boldsymbol{\tau})\lesssim p^{\frac{n+1}{2-\gamma}}$ respectively, for an arbitrarily small constant $\gamma\in(0,1)$ and arbitrary constant $n>0$. This results in the corresponding operator network $\mathscr{N}$ satisfying the given estimate:
    \begin{equation*}
            \mathbb{E}|\mathscr{E}_{N,J}(\mathscr{N})|\lesssim \widehat{\mathscr{E}}(\mathscr{N})/N^{1/4}+p^{-n}+\frac{C\sqrt{\ln{J}}}{mJ^{1/2}},
    \end{equation*}
where $C$ is a constant related only to $\alpha$ and $l$. Here, $\widehat{\mathscr{E}}(\mathscr{N})$ is given in Theorem \ref{cor_approximation error}.
\label{empirical risk thm}
\end{theorem}
\subsection{Proof Overview}

This section presents the proofs for Theorem \ref{generelization gap} and Theorem \ref{empirical risk thm}. To simplify notation, we define two operators: $\Delta_{\mathscr{N}}$ and $\mathcal{I}$. The former is defined as $\Delta_{\mathscr{N}}(f)(y)=|\mathscr{G}(f)(y)-\mathscr{N}(f)(y)|^2$, while the latter is defined as $\mathcal{I}(g)=\sum\limits_{j=0}^J\omega_jg(y_j)$. By assuming that the operator $\mathscr{N}$ satisfies Assumption \ref{assump_2}, we can infer the existence of an operator $\widetilde{\mathscr{N}}$ that satisfies inequalities \eqref{equ_assump_2} and \eqref{equ_assump_3}. Using this, we can control the difference $\mathscr{E}_{diff}(\mathscr{N})$ with the following four terms:
        \begin{align}
\mathscr{E}_{diff}(\mathscr{N})=&\mathbb{E}\left|\widehat{\mathscr{E}}(\mathscr{N})-\mathscr{E}_{N,J}(\mathscr{N})\right|\nonumber\\
        \le&\mathbb{E}\left|\left(\mathbb{E}_{f\sim\mu}\|\Delta_{\mathscr{N}}(f)\|_{L^1}\right)^{1/2}-\left(\frac{1}{N}\sum\limits_{n=1}^N\|\Delta_{\mathscr{N}}(F_n)\|_{L^1}\right)^{1/2}\right|+\mathbb{E}\left|\left(\frac{1}{N}\sum\limits_{n=1}^N\|\Delta_{\mathscr{N}}(F_n)\|_{L^1}\right)^{1/2}\right.\nonumber\\
        &\left.-\left(\frac{1}{N}\sum\limits_{n=1}^N\|\Delta_{\widetilde{\mathscr{N}}}(F_n)\|_{L^1}\right)^{1/2}\right|+\mathbb{E}\left|\left(\frac{1}{N}\sum\limits_{n=1}^N\|\Delta_{\widetilde{\mathscr{N}}}(F_n)\|_{L^1}\right)^{1/2}-\left(\frac{1}{N}\sum\limits_{n=1}^N\mathcal{I}(\Delta_{\widetilde{\mathscr{N}}}(F_n))\right)^{1/2}\right|\nonumber\\
        &+\mathbb{E}\left|\left(\frac{1}{N}\sum\limits_{n=1}^N\mathcal{I}(\Delta_{\widetilde{\mathscr{N}}}(F_n))\right)^{1/2}-\left(\frac{1}{N}\sum\limits_{n=1}^N\mathcal{I}(\Delta_{\mathscr{N}}(F_n))\right)^{1/2}\right|\nonumber\\
        \triangleq&\text{\uppercase\expandafter{\romannumeral1}}+\text{\uppercase\expandafter{\romannumeral2}}+\text{\uppercase\expandafter{\romannumeral3}}+\text{\uppercase\expandafter{\romannumeral4}}.\label{diff_decomposition}
        \end{align}
This inequality will serve as the basis for proving both Theorem \ref{generelization gap} and Theorem \ref{empirical risk thm}.

\subsubsection{Proof of Theorem \ref{generelization gap}}

By leveraging this observation, we can now establish a rigorous proof for Theorem \ref{generelization gap}. To begin, we estimate the first term \uppercase\expandafter{\romannumeral1} in \eqref{diff_decomposition} as follows:
\begin{lemma}
Let $\mathscr{G}$ be the solution operator of Eq.~\eqref{equ} and let the operator $\mathscr{N}$ satisfies Assumption \ref{assump_1}. Then, we have 
\begin{equation*}
    \text{\uppercase\expandafter{\romannumeral1}}\lesssim 1/N^{1/4}.
\end{equation*}
\label{lemma_decom_1}
\end{lemma}
\begin{proof}
       Let $X_n=\|\mathscr{G}(F_n)-\mathscr{N}(F_n)\|_{L^2}^2$ and $S^N=\frac{1}{N}\sum\limits_{n=1}^NX_n$, we can get
\begin{align}
\text{\uppercase\expandafter{\romannumeral1}}&=\mathbb{E}\left|\sqrt{S^N}-\sqrt{\mathbb{E}S^N}\right|\le\mathbb{E}\sqrt{|S^N-\mathbb{E}S^N|}\nonumber\\
&\le\left(\mathbb{E}|S^N-\mathbb{E}(S^N)|^2\right)^{1/4}=\frac{1}{\sqrt{N}}\left(\mathbb{E}\left(\sum\limits_{n=1}^N(X_N-\mathbb{E}X_N)\right)^2\right)^{1/4}\nonumber\\
&=\frac{1}{\sqrt{N}}\left(\mathbb{E}\left(\sum\limits_{n=1}^N(X_n-\mathbb{E}X_n)^2+\sum\limits_{m\ne n}(X_n-\mathbb{E}X_n)(X_m-\mathbb{E}X_m)\right)\right)^{1/4}\nonumber\\
&=\frac{1}{\sqrt{N}}\left(\sum\limits_{n=1}^N\mathbb{E}(X_n-\mathbb{E}X_n)^2\right)^{1/4}=\frac{1}{N^{1/4}}\left(\mathbb{E}(X_1-\mathbb{E}X_1)^2\right)^{1/4}\nonumber\\
&\le \frac{1}{N^{1/4}}\left((\mathbb{E}X_1^2)^{1/4}+(\mathbb{E}X_1)^{1/2}\right).\label{decom_1}
\end{align}
The boundedness of $\mathbb{E}X_1^2$, as expressed in Assumption \ref{assump_1}, has already been established. Meanwhile, the result of Theorem \ref{cor_approximation error} directly yields the boundedness of the other term, $\mathbb{E}X_1$, then we can directly obtain the desired conclusion from the above inequality \eqref{decom_1}.
\qed
\end{proof}

Exploiting the relationship captured by inequality \eqref{equ_assump_2} between the operators $\widetilde{\mathscr{N}}$ and $\mathscr{N}$, we can arrive at a useful estimate by simply manipulating terms \uppercase\expandafter{\romannumeral2} and \uppercase\expandafter{\romannumeral4} in Eq.~\eqref{diff_decomposition}:
\begin{equation} 
\text{\uppercase\expandafter{\romannumeral2}}+\text{\uppercase\expandafter{\romannumeral4}}\le (\mathbb{E}_{f\sim\mu}\|\mathscr{N}(f)-\widetilde{\mathscr{N}}(f)\|_{L^2}^2)^{1/2}+(\sum\limits_{j=0}^J\omega_j\mathbb{E}_{f\sim\mu}|\mathscr{N}(f)(y_j)-\widetilde{\mathscr{N}}(f)(y_j)|^2)^{1/2}\lesssim1/N^{1/4}.\label{2+4}
\end{equation}

The error denoted by \uppercase\expandafter{\romannumeral3} in \eqref{diff_decomposition} arises from employing a numerical integration method to estimate the integral over the interval $[0,1]$. This can be bounded by:
        \begin{align}
        \text{\uppercase\expandafter{\romannumeral3}}&=\mathbb{E}\left|\left(\frac{1}{N}\sum\limits_{n=1}^N\|\mathscr{G}(F_n)-\widetilde{\mathscr{N}}(F_n)\|_{L^2}^2\right)^{1/2}-\left(\frac{1}{N}\sum\limits_{n=1}^N\sum\limits_{j=0}^J\omega_j|\mathscr{G}(F_n)(y_j)-\widetilde{\mathscr{N}}(F_n)(y_j)|^2\right)^{1/2}\right|\nonumber\\
        &\le \left(\mathbb{E}\left|\frac{1}{N}\sum\limits_{n=1}^N\left(\|\mathscr{G}(F_n)-\widetilde{\mathscr{N}}(F_n)\|_{L^2}^2-\sum\limits_{j=0}^J\omega_j|\mathscr{G}(F_n)(y_j)-\widetilde{\mathscr{N}}(F_n)(y_j)|^2\right)\right|\right)^{1/2}\nonumber\\
        & \le \left(\mathbb{E}_{f\sim\mu}\left|\|\mathscr{G}(f)-\widetilde{\mathscr{N}}(f)\|_{L^2}^2-\sum\limits_{j=0}^J\omega_j|\mathscr{G}(f)(y_j)-\widetilde{\mathscr{N}}(f)(y_j)|^2\right|\right)^{1/2}.\label{term3_decom}
        \end{align}
To compute this error term, we present a lemma in this regard, and the proof of this lemma is outlined in Appendix \ref{appendix 7.1}.
\begin{lemma}
    Suppose $F:L^2([0,1])\to C^1([0,1])$ is an operator satisfying $\mathbb{E}_{f\sim\mu}|\frac{d}{dx}F(f)(x)| \le C (1+\varepsilon^{-1} e^{-\alpha(1-x)/\varepsilon})$ for $x \in [0,1]$, where $\alpha>0$ is a constant. Assuming a piece-wise uniform mesh with $J+1$ points given by \eqref{mesh} as $\{x_j\}_{j=0}^J$, we have
    \begin{equation*}
        \mathbb{E}_{f\sim\mu}\left|\int_0^1F(f)(x)dx-\sum\limits_{j=0}^{J-1}h_jF(f)(x_j)\right|\le 2C(1+2/\alpha)\frac{\ln J}{J},
    \end{equation*}
where $h_j=x_{j+1}-x_j$.
\label{expect_interpolation_error}
\end{lemma}

By setting $F=(\mathscr{G}-\widetilde{\mathscr{N}})^2$ in Lemma \ref{expect_interpolation_error}, we can directly obtain an estimate for term \uppercase\expandafter{\romannumeral3} using the conclusion of this lemma, given that inequality \eqref{equ_assump_3} is satisfied. Specifically, we have 
\begin{equation}
    \text{\uppercase\expandafter{\romannumeral3}}\lesssim\left(\ln{J}/J\right)^{1/2}.\label{3}
\end{equation}
Combining Lemma \ref{lemma_decom_1} with formulas \eqref{2+4}, \eqref{3}, we arrive at the conclusion of Theorem \ref{generelization gap}.

\subsubsection{Proof of Theorem \ref{empirical risk thm}}

In Section \ref{sec_approximation}, we construct operators in the form of \eqref{deeponet} and estimate the error incurred by these operators in approximating the target operator $\mathscr{G}$. Of particular interest is the operator defined as
\begin{equation}
    \hspace{0.5em}\overline{\kern-0.5em\mathscr{N}\kern+0.1em}\hspace{-0.1em}(f)=\mathbb{E}_{\nu}+\sum\limits_{j=-K}^K(\mathscr{G}\circ\mathscr{D}\circ\mathscr{E}(f),\widetilde{\tau}_j)\tau_j,
    \label{operator_net}
\end{equation}
where operator $\mathscr{D}$ is detailed in Remark \ref{rek_encoder}, and $\widetilde{\tau}_j$, $j=-K,\dots,K$ is obtained by orthogonalizing the set of functions $\{\mathscr{G}(\phi_j)=\mathscr{G}(e^{i2\pi kx})\}_{j=-K}^K$. Additionally, $\{\tau^0, \{\tau_j\}_{j=-K}^K\}$ is a set of neural networks that satisfies the condition:
\begin{equation*}
    \max\{\|\tau^0-\mathbb{E}_{\nu}\|_{L^{\infty}},\mathop {\max }\limits_{j=-K,\dots,K}\|\tau_j-\widetilde{\tau}_j\|_{L^{\infty}}\}<\delta=p^{-(n+1)},
\end{equation*}
where $n$ is a fixed, arbitrary positive constant. 

In the following, we set out to investigate the empirical risk $\mathscr{E}_{N,J}(\hspace{0.5em}\overline{\kern-0.5em\mathscr{N}\kern+0.1em}\hspace{-0.1em})$ associated with the operator at hand, and provide proof of Theorem \ref{empirical risk thm}.  While Theorem \ref{cor_approximation error} has previously addressed the estimation of $\widehat{\mathscr{E}}(\hspace{0.5em}\overline{\kern-0.5em\mathscr{N}\kern+0.1em}\hspace{-0.1em})$, our present focus centers on analyzing $\mathscr{E}_{diff}(\hspace{0.5em}\overline{\kern-0.5em\mathscr{N}\kern+0.1em}\hspace{-0.1em})$. Considering that the activation function of $\hspace{0.5em}\overline{\kern-0.5em\mathscr{N}\kern+0.1em}\hspace{-0.1em}$ is not always differentiable, we introduce a modified operator $\widetilde{\mathscr{N}}$, which is obtained from $\hspace{0.5em}\overline{\kern-0.5em\mathscr{N}\kern+0.1em}\hspace{-0.1em}$ by replacing the bases function. Specifically, $\widetilde{\mathscr{N}}(f)$ is defined as:
\begin{equation}
    \widetilde{\mathscr{N}}(f)=\tau^0+\sum\limits_{j=-K}^K(\mathscr{G}\circ\mathscr{D}\circ\mathscr{E}(f),\widetilde{\tau}_j)\widetilde{\tau}_j.
    \label{tilde_operator_net}
\end{equation}

By substituting $\widetilde{\mathscr{N}}$ into Eq.~\eqref{diff_decomposition}, we can control $\mathscr{E}_{diff}(\hspace{0.5em}\overline{\kern-0.5em\mathscr{N}\kern+0.1em}\hspace{-0.1em})$ via the four terms associated with $\hspace{0.5em}\overline{\kern-0.5em\mathscr{N}\kern+0.1em}\hspace{-0.1em}$ and $\widetilde{\mathscr{N}}$, still denoted as \uppercase\expandafter{\romannumeral1}, \uppercase\expandafter{\romannumeral2}, \uppercase\expandafter{\romannumeral3}, and \uppercase\expandafter{\romannumeral4}. The estimation of terms \uppercase\expandafter{\romannumeral1}, \uppercase\expandafter{\romannumeral2}, and \uppercase\expandafter{\romannumeral4} is straightforward, and we present the results in Lemma \ref{lem_diff124} (with the proof given in Appendix \ref{appendix 7.2}).
\begin{lemma}
Assuming $m \leq p$, we can consider the operators $\hspace{0.5em}\overline{\kern-0.5em\mathscr{N}\kern+0.1em}\hspace{-0.1em}$ and $\widetilde{\mathscr{N}}$, as defined in Eq.~\eqref{operator_net} and Eq.~\eqref{tilde_operator_net}, respectively. For the terms $\uppercase\expandafter{\romannumeral1}$, $\uppercase\expandafter{\romannumeral2}$, and $\uppercase\expandafter{\romannumeral4}$ in Eq.~\eqref{diff_decomposition}, an estimate for their sum can be derived as follows:
\begin{equation*}
\uppercase\expandafter{\romannumeral1}+\uppercase\expandafter{\romannumeral2}+\uppercase\expandafter{\romannumeral4}\lesssim\widehat{\mathscr{E}}(\hspace{0.5em}\overline{\kern-0.5em\mathscr{N}\kern+0.1em}\hspace{-0.1em})/N^{1/4}+p^{-n},
\end{equation*}
where $\widehat{\mathscr{E}}(\hspace{0.5em}\overline{\kern-0.5em\mathscr{N}\kern+0.1em}\hspace{-0.1em})$ is estimated using Theorem \ref{cor_approximation error}.\label{lem_diff124}
\end{lemma}
    
   We now focus on the term \uppercase\expandafter{\romannumeral3}, which can be bounded using inequality \eqref{term3_decom}.    We define $F(f)=(\mathscr{G}(f)-\widetilde{\mathscr{N}}(f))^2$, and according to Lemma \ref{expect_interpolation_error}, estimating term \uppercase\expandafter{\romannumeral3} requires only an estimation of  $\mathbb{E}_{f\sim \mu}\left|\frac{d}{dx}F(f)(x)\right|$. We focus on this quantity next. In fact, when $M\le K$ (that is, when $m\le p$), we have the following simple form for $\mathscr{G}(f)-\widetilde{\mathscr{N}}(f)$:
\begin{equation}
    \mathscr{G}(f)(x)-\widetilde{\mathscr{N}}(f)(x)=\mathscr{G}(f)(x)-\mathscr{G}\circ\mathscr{D}\circ\mathscr{E}(f)(x).\label{simplify}
\end{equation}
Here, we have used the fact that when $m\le p$,
\begin{equation*}
    \widetilde{\mathscr{N}}(f)=\sum\limits_{j=-K}^K(\mathscr{G}\circ\mathscr{D}\circ\mathscr{E}(f),\widetilde{\tau}_j)\widetilde{\tau}_j=\mathscr{G}\circ\mathscr{D}\circ\mathscr{E}(f).
\end{equation*}
With this observation, we obtain the estimate for $\mathbb{E}_{f\sim \mu}\left|\frac{d}{dx}F(f)(x)\right|$ using the following lemma (proof provided in Appendix \ref{appendix 7.3}).
\begin{lemma}
    When $m\le p$, the following inequality holds:
    \begin{equation*}
        \mathbb{E}_{f\sim\mu}\left|\left(\mathscr{G}(f)(x)-\widetilde{\mathscr{N}}(f)(x)\right)\frac{d}{dx}\left(\mathscr{G}(f)(x)-\widetilde{\mathscr{N}}(f)(x)\right)\right|\le\frac{C}{m^2}(1+1/\varepsilon e^{-\alpha(1-x)/\varepsilon}).
    \end{equation*}
\label{deri_bound}
\end{lemma}  

By leveraging Lemma \ref{expect_interpolation_error} and Lemma \ref{deri_bound}, we derive an upper bound for the term \uppercase\expandafter{\romannumeral3} in Eq.~\eqref{diff_decomposition}:
\begin{equation*}
\text{\uppercase\expandafter{\romannumeral3}}\lesssim\frac{\sqrt{\ln{J}}}{mJ^{1/2}}.
\end{equation*}
Combining this estimate with the result of Lemma \ref{lem_diff124} enables us to complete the proof of Theorem \ref{empirical risk thm}.
\section{Numerical Experiments\label{sec_experiemnt}}

In this section, we demonstrate the effectiveness of DeepONets in solving singularly perturbed problems through three concrete examples, thereby contributing to the growing body of theoretical understanding of this model. The model's trainable parameters in Eq.~\eqref{deeponet}, denoted by $\boldsymbol{\theta}$, are obtained through the minimization of a loss function defined as follows:
\begin{equation}
    \mathcal{L}(\boldsymbol{\theta})=\frac{1}{N(J+1)}\sum\limits_{n=1}^{N}\sum\limits_{j=0}^{J}|u_n(y_j)-\mathscr{N}_{\boldsymbol{\theta}}(F_n)(y_j)| ^2.\label{loss}
\end{equation}
Here, $u_n$ represents the ground truth, which is typically obtained through numerical methods of high precision. To obtain $u_n(y_j)$, we utilize an up-winding scheme on the Shishkin mesh (see \cite{steady_convection_diffusion}) and interpolate accordingly.

$N$ samples, denoted $F_1,F_2,\dots,F_N$, are drawn independently and identically from the common distribution $\mu$. Examples of several pairs $(F_i,u_i)$ are provided in Appendix \ref{appendix 7.4}. In general, locations $\{y_i\}$ can be selected randomly from the domain $[0,1]$ or arranged as equidistant points. However, as noted previously, due to the presence of a thin boundary layer, the use of these sampling methods would demand a large number of points to ensure that the trained operator network can capture the relevant boundary layer information accurately. Conversely, sampling a large number of points far away from the boundary layer is unnecessary since the solution's behavior in this region is regular and slow. Consequently, we adopt Shishkin mesh points, defined in \eqref{mesh}, as the locations in \eqref{loss} within the interval $[0,1]$.
\begin{remark}
   Although machine learning models are commonly obtained through empirical risk minimization, where $\boldsymbol{\theta}$ is chosen to minimize empirical risk, and an operator network is determined accordingly. Notably, the solution to the singularly perturbed problem \eqref{equ} can vary across different length scales. To address this issue, the loss function \eqref{loss} assigns all weights $\omega_j$, where $j=0,1,\dots,J$ in \eqref{empirical risk}, a uniform value of $1/(J+1)$. This rescales the microscale and places data at different scales on the same scale. This approach is expected to mitigate the spectral bias phenomenon. The numerical results presented in the following sections confirm this conjecture.
\end{remark}
           
Throughout the following examples, we set the positive lower bound $\alpha$ of the function $p(x)$ to 1. Then the transition point of the Shishkin mesh is $\sigma=\min\{1/2,\ 2\varepsilon\ln{J}\}$.

\begin{example}
\label{exam_1}
We consider a boundary layer problem as presented in \cite{steady_convection_diffusion}:
\begin{equation}
    \left\{ \begin{array}{l}
-\varepsilon u''+u'=f,\quad x\in(0,1),\\
u(0)=u(1)=0.
\end{array} \right.\label{exam1}
\end{equation}
The Lipschitz continuity of the solution operator that maps $f(x)$ to the ODE solution $u(x)$ is proved in Appendix \ref{appendix 7.5}. Subsequently, we utilize DeepONets to learn this operator. To generate the necessary training dataset, we solve the equation using an up-winding scheme on a Shishkin mesh of 4096 points, and obtain $\{u(y_j)\}_{j=0}^J$ via interpolation.

We begin by training a DeepONet using 1000 random $f$ samples and 256 $y$ locations (i.e., $N=1000$ and $J=256$), resulting in a training dataset of size $2.56\times 10^5$. The training is conducted over 1000 epochs. Fig.~\ref{fig1} illustrates the predictive capability of the trained model, displaying its ability to predict solutions for two distinct $f$. The first is randomly sampled from $\mu$ with the parameter $l=1$, while the second is a simple out-of-distribution signal given by $f(x)\equiv1$.
\begin{figure}[!tbh]
\centering
\begin{minipage}{0.325\textwidth}
\centering
\includegraphics[width=\textwidth]{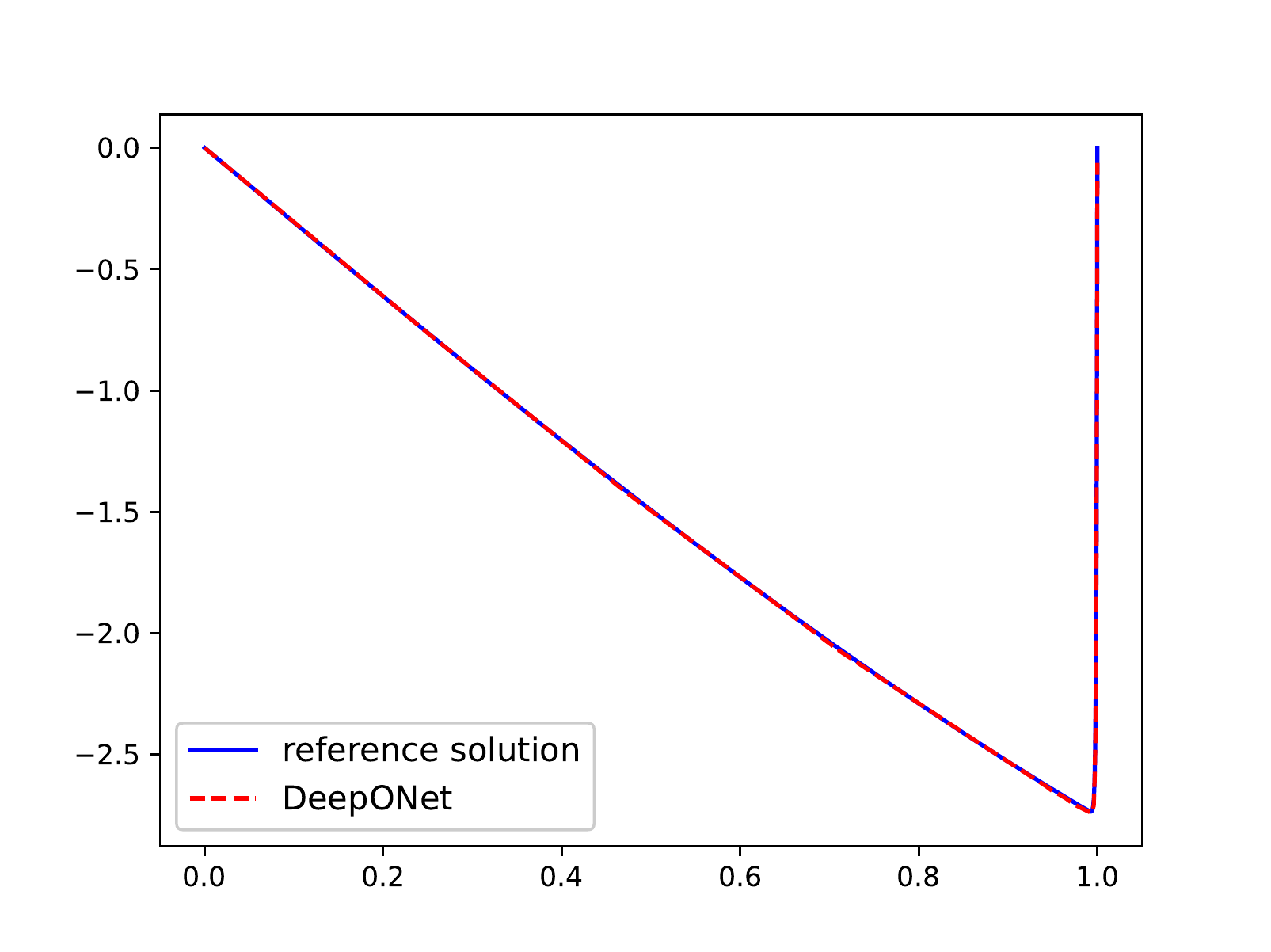}\\
\scriptsize{a)}
\end{minipage}
\begin{minipage}{0.325\textwidth}
\centering
\begin{minipage}{\textwidth}
\centering
\includegraphics[width=\textwidth]{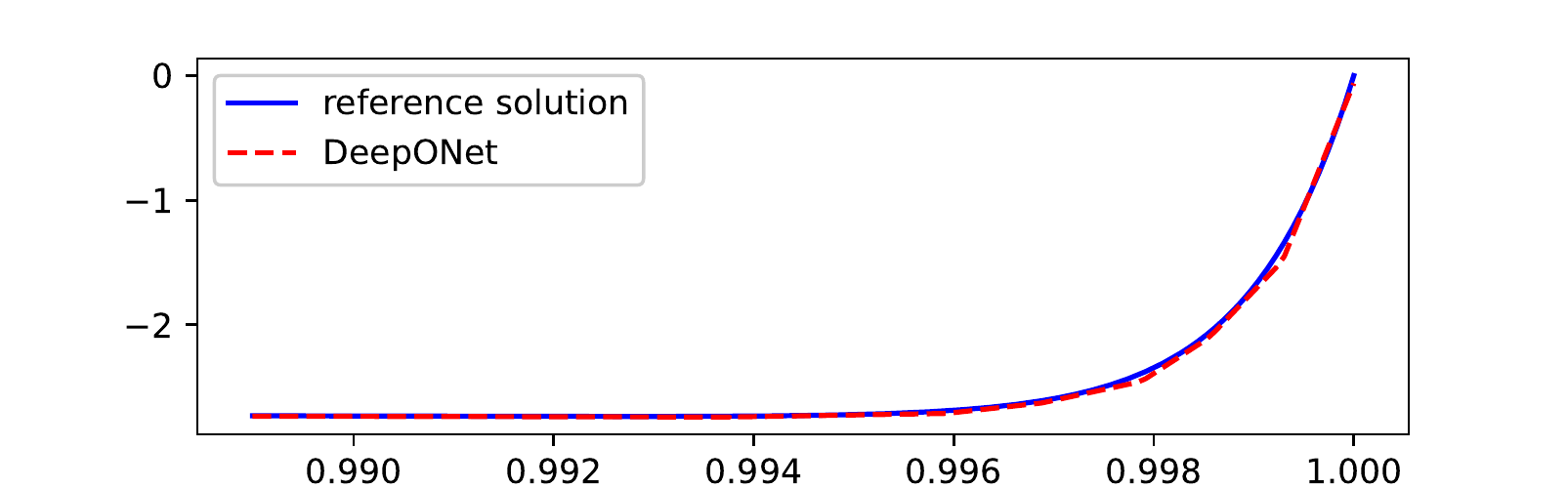}\\
\scriptsize{b)}
\end{minipage}
\\
\begin{minipage}{\textwidth}
\centering
\includegraphics[width=\textwidth]{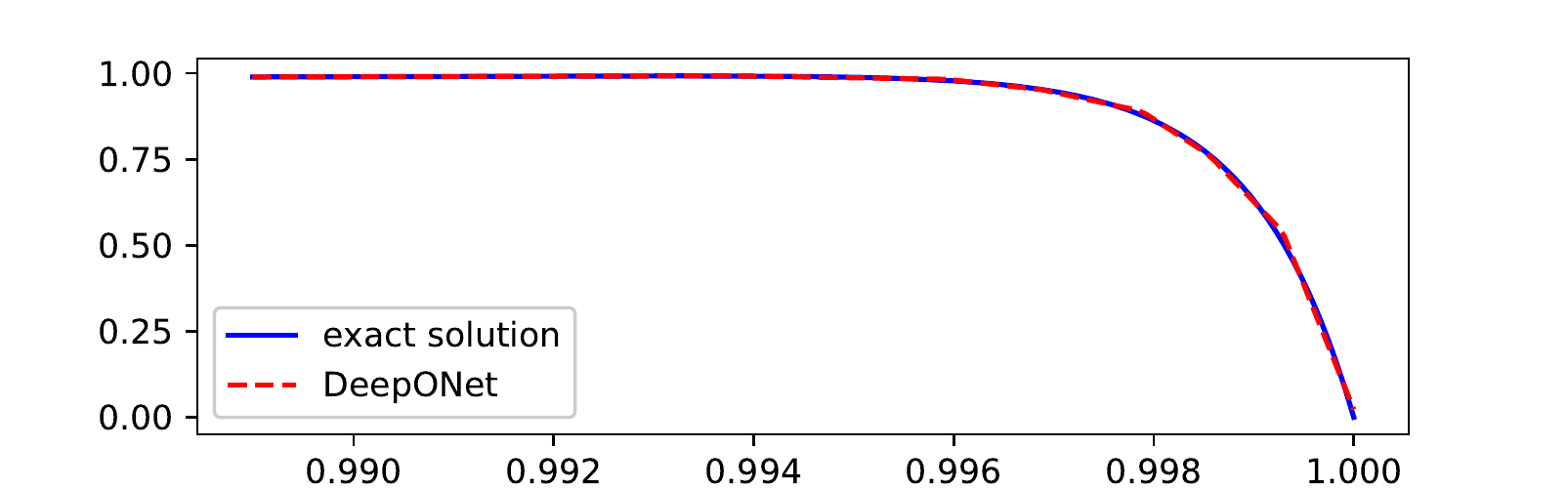}\\
\scriptsize{c)}
\end{minipage}
\end{minipage}
\begin{minipage}{0.325\textwidth}
\centering
\includegraphics[width=\textwidth]{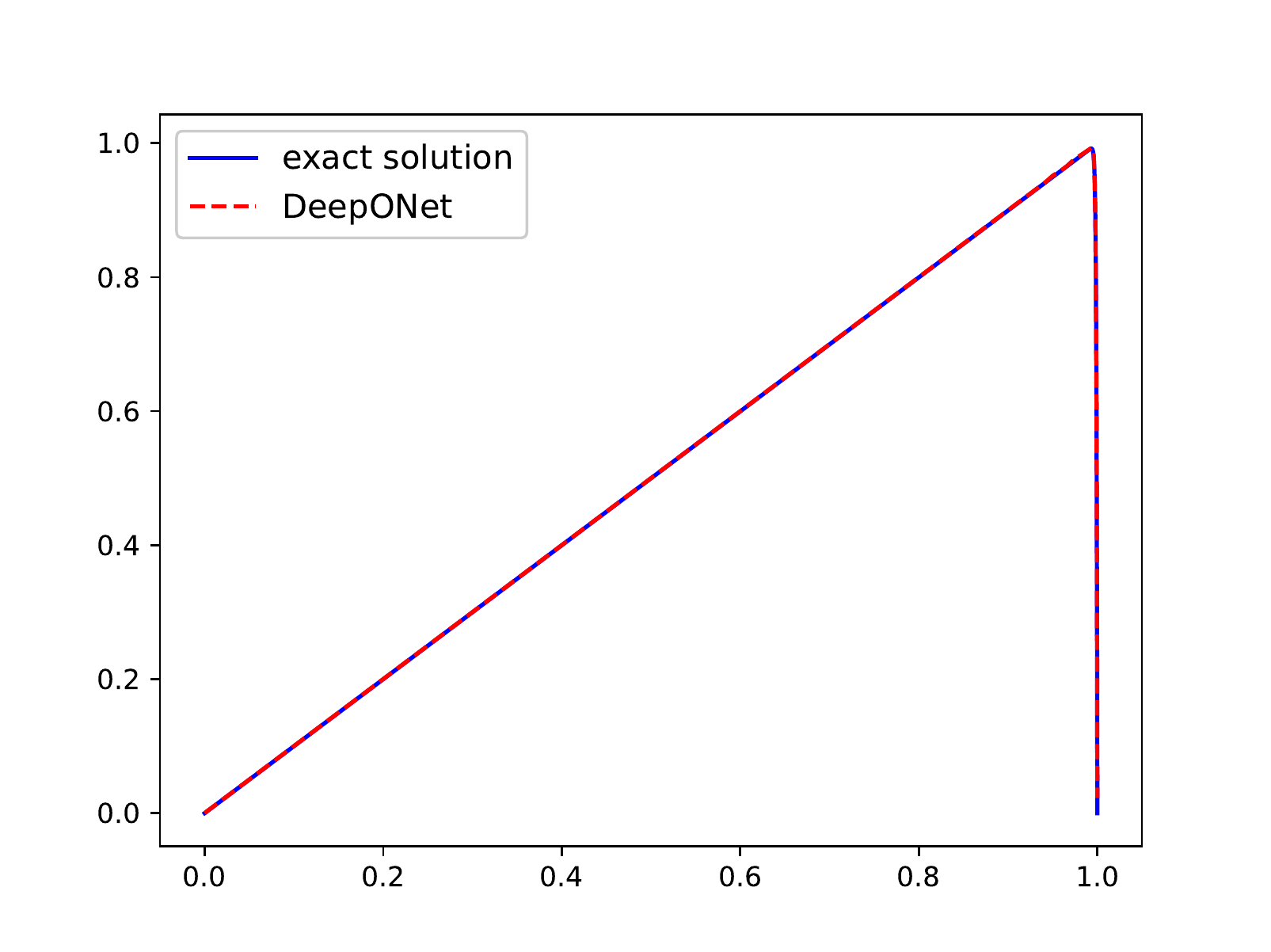}\\
\scriptsize{d)}
\end{minipage}
\caption{The trained model's predictions and reference solutions for Eq.~\eqref{exam1} ($\varepsilon=0.001$) on Shishkin mesh. a): Prediction and reference solution on a random $f\sim\mu$ on a grid of 256 points, with the latter using an up-winding scheme on Shishkin mesh. b): Horizontal magnification of the region $[1-\sigma,1]$ in a). c): Horizontal magnification of the region $[1-\sigma,1]$ in d). d): the model's prediction and exact solution for $f(x)\equiv1$ on a grid of 256 points, the exact solution is given by $u(x)=x-(e^{-(1-x)/\varepsilon}-e^{-1/\varepsilon})/(1-e^{-1/\varepsilon})$.}
\label{fig1}
\end{figure}

\begin{figure}[!tbh]
\centering
\begin{minipage}{0.325\textwidth}
\centering
\includegraphics[width=\textwidth]{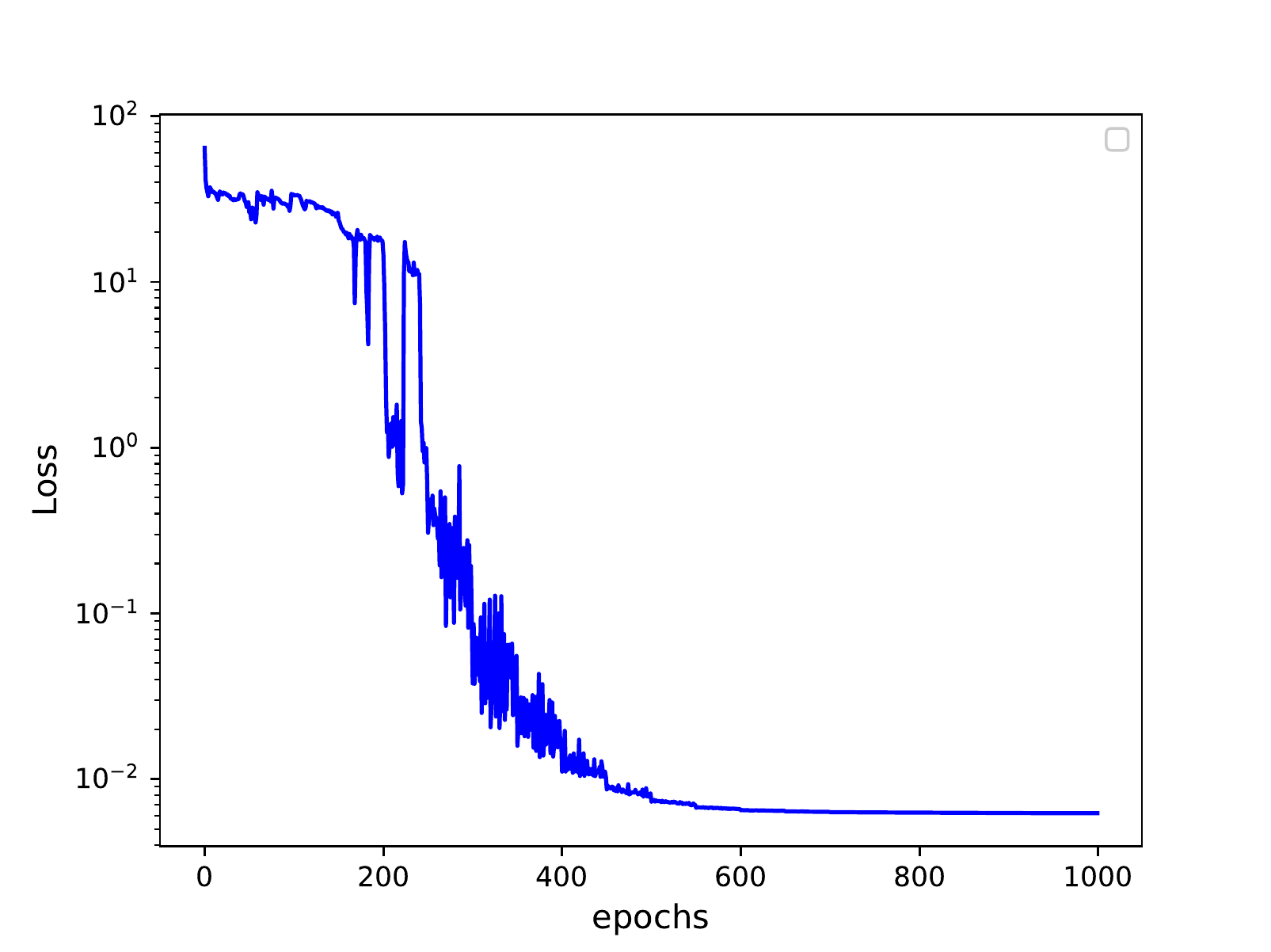}\\
\scriptsize{a)}
\end{minipage}
\begin{minipage}{0.325\textwidth}
\centering
\includegraphics[width=\textwidth]{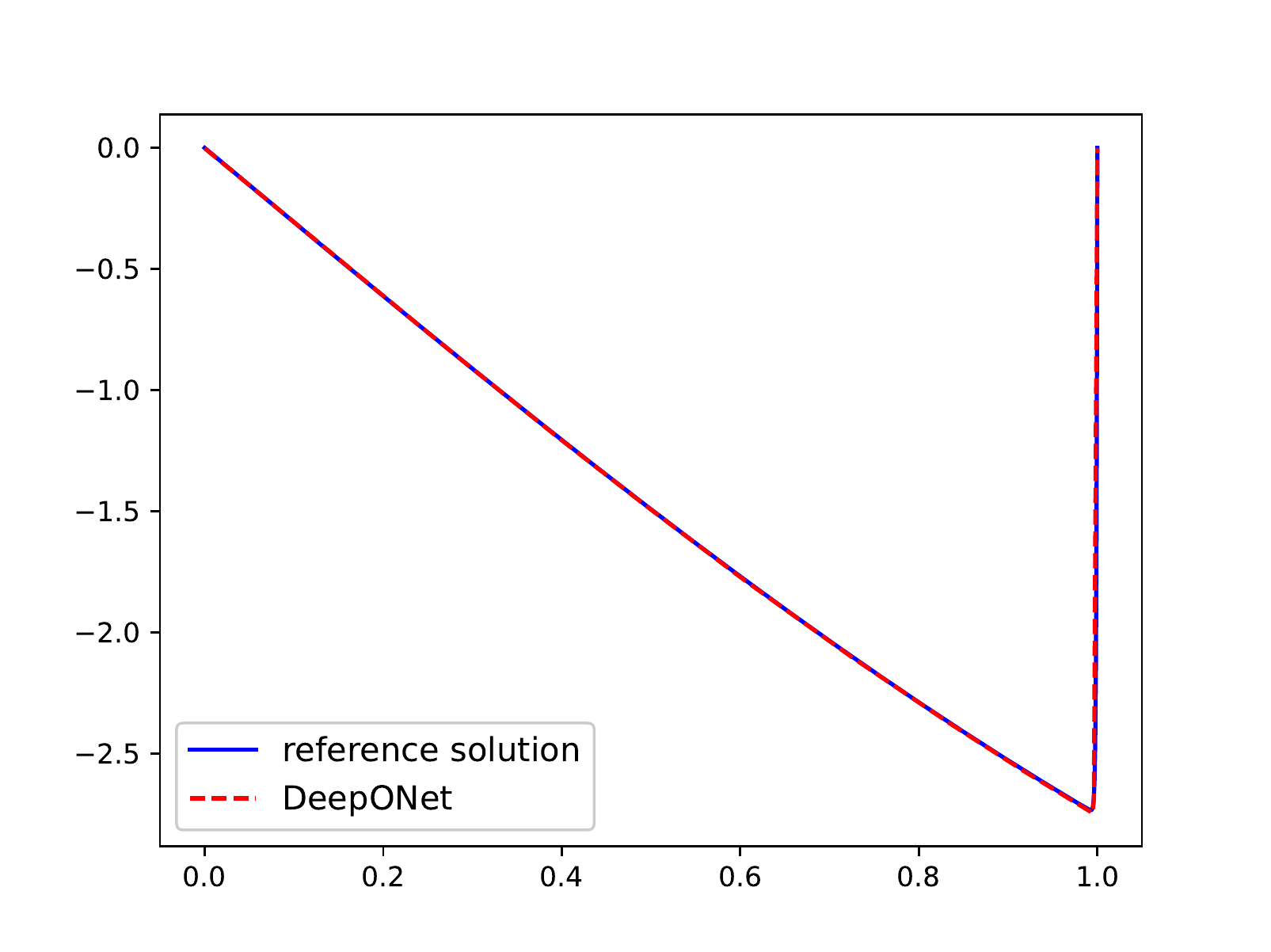}\\
\scriptsize{b)}
\end{minipage}
\begin{minipage}{0.325\textwidth}
\centering
\includegraphics[width=\textwidth]{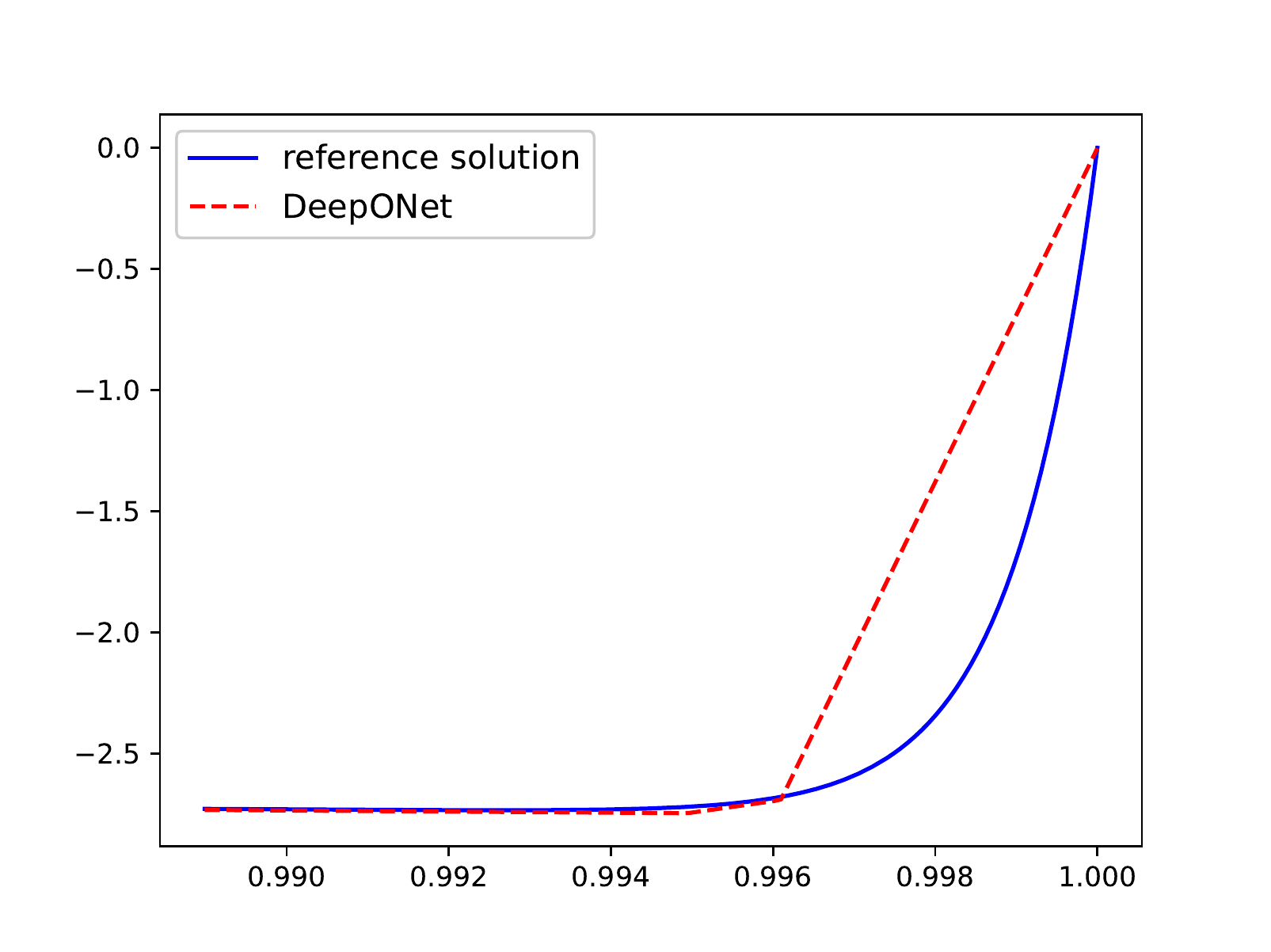}\\
\scriptsize{c)}
\end{minipage}
 \caption{Performance of a trained model using equidistant location points ($\varepsilon=0.001$). a):  The training trajectory of a DeepONet with equidistant location points. Here “epochs” refers to the number of steps that the Adam optimizer performs. b): Prediction using the trained model and reference solution on a random $f\sim\mu$ on a Shishkin mesh of 256 points. c): Horizontal magnification of the region $[1-\sigma,1]$ in b).  }
 \label{fig2}
\end{figure}

As previously noted, accurate capture of boundary layer information using a trained model necessitates taking a sufficient number of data points, particularly when sampling at equidistant locations. To validate this conjecture, we conducted experimental tests. To begin with, we train a DeepONet leveraging 1000 random $f$ samples and 256 equidistant locations $y$. In Figure~\ref{fig2}, subplot a) displays the trajectory of the training process. Subplots b) and c) indicate that training the network using equidistant locations is effective in capturing information beyond the boundary layer; however, it is not as effective in dealing with behavior within the boundary layer.

By utilizing the definition of the generalization gap, empirical risk can serve as an effective tool for assessing a model's generalization ability. A comparison of Fig.~\ref{fig1} and Fig.~\ref{fig2} reveals that the network trained on equidistant locations is inferior to the network trained on Shishkin mesh points in describing the boundary layer behavior. Fig.~\ref{fig3} provides a more comprehensive analysis, using both 1000 $f$ samples. The comparison between subplots a) and b) in Fig.\ref{fig3} shows that the network trained on Shishkin mesh points exhibits consistent generalization ability across various values of $\varepsilon$. Conversely, the network trained on equidistant points demonstrates a significant correlation with $1/\varepsilon$, indicating a reduction in its generalization ability as $\varepsilon$ decreases. Moreover, these figures demonstrate that the model's generalization ability improves with the increase in the number of locations $y$, i.e., $J$.
\begin{figure}[!tbh]
\centering
\begin{minipage}{0.4\textwidth}
\centering
\includegraphics[width=\textwidth]{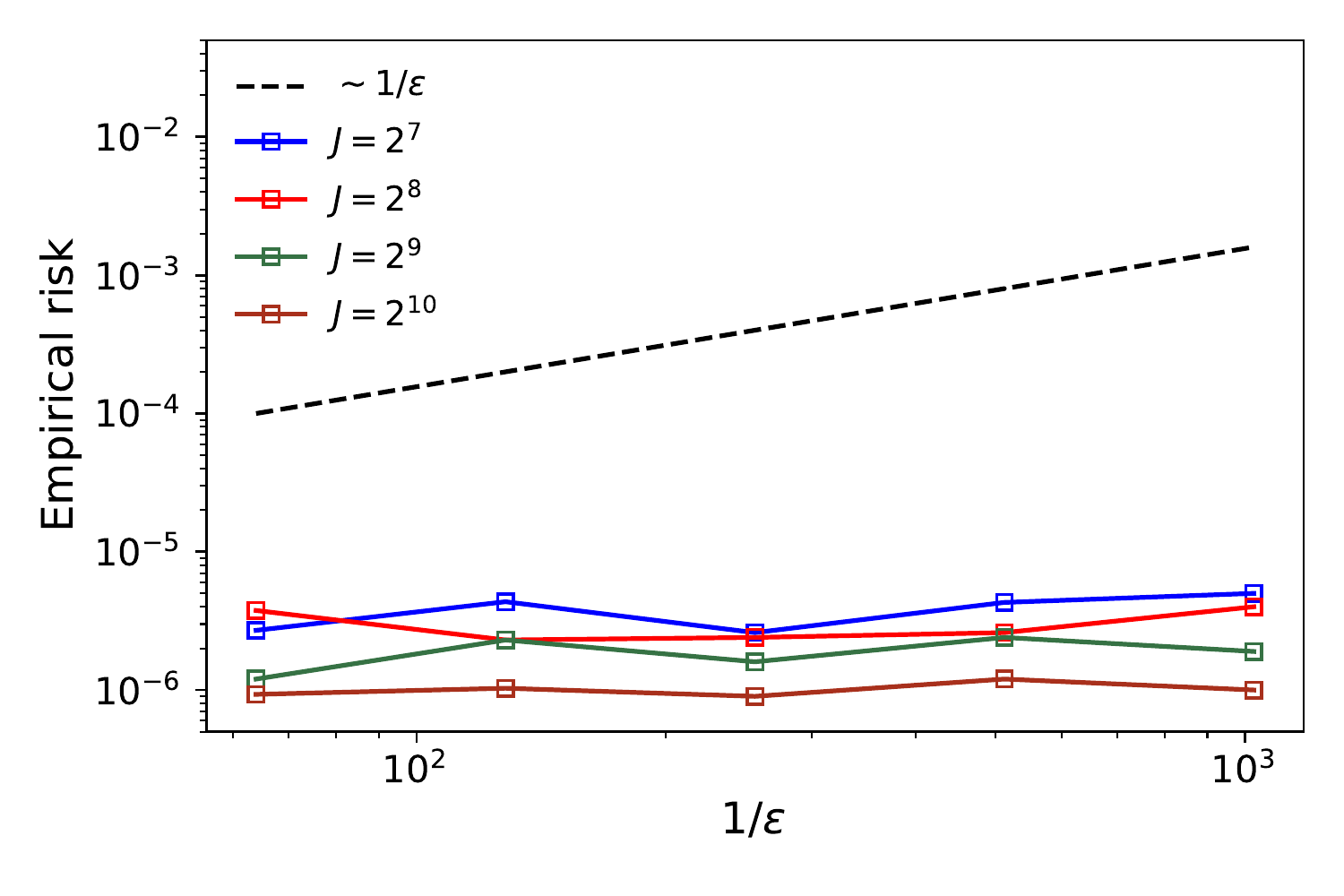}\\
\scriptsize{a)}
\end{minipage}
\begin{minipage}{0.4\textwidth}
\centering
\includegraphics[width=\textwidth]{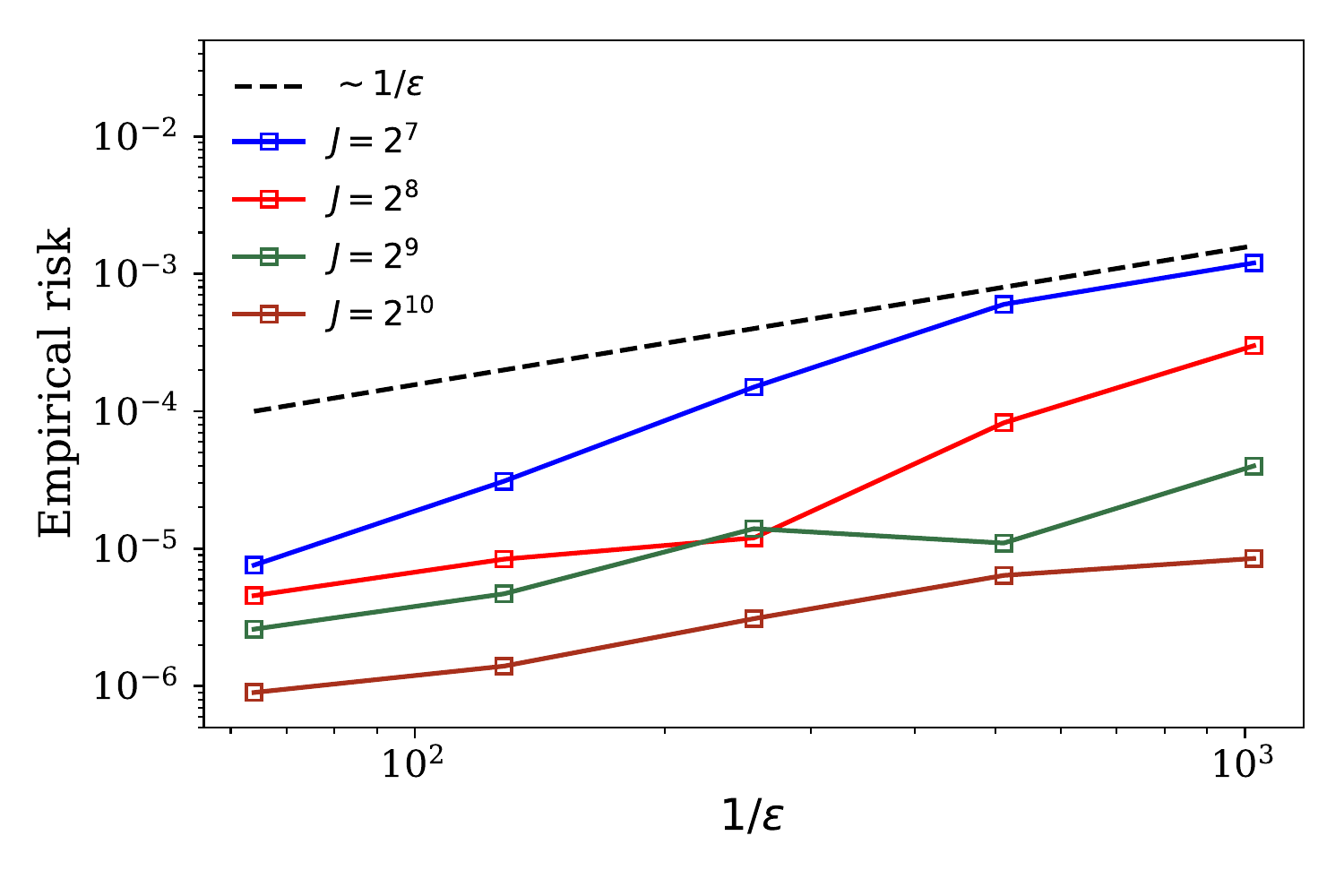}\\
\scriptsize{b)}
\end{minipage}
  \caption{Models trained on a different number of Shishkin mesh points and equidistant points are compared in terms of empirical risk, evaluated on a $2^{12}$ location point mesh. The subplots display empirical risk against the parameter ($\varepsilon$) on a $1/\varepsilon$ scale, with the dashed black line running parallel to $1/\varepsilon$. a): The empirical risk performance of models trained on different numbers of Shishkin mesh points. b): The performance of models trained on different numbers of equidistant points.}
 \label{fig3}
\end{figure}
\end{example}
\begin{example}
Consider a singularly perturbed problem with variable coefficients given by:
\begin{equation}
    \left\{ \begin{array}{l}
-\varepsilon u''+(x+1)u'+u=f,\quad x\in(0,1),\\
u(0)=u(1)=0.
\end{array} \right.
\label{exam2}
\end{equation}

The solution operator mapping from $f$ to the solution $u$ is known to be Lipshitz, according to Lemma \ref{lem_lip}. We use DeepONets to approximate this operator. We construct a DeepONet by training it on a dataset comprising $2.56\times 10^5$ triplets $(f,y,u)$ over 1000 epochs. The input data consists of 1000 randomly sampled $f$ and 256 distinct locations $y$ (specifically, a Shishkin mesh of 256 points), while the corresponding output data, $u$, is obtained using an up-wind scheme on the Shishkin mesh. Fig.~\ref{fig4} demonstrates the network's ability to accurately estimate the output solution $u$, not only for a randomly sampled input $f$ but also for an out-of-distribution input $f(x)=e^x$.

\begin{figure}[!tbh]
\centering
\begin{minipage}{0.325\textwidth}
\centering
\includegraphics[width=\textwidth]{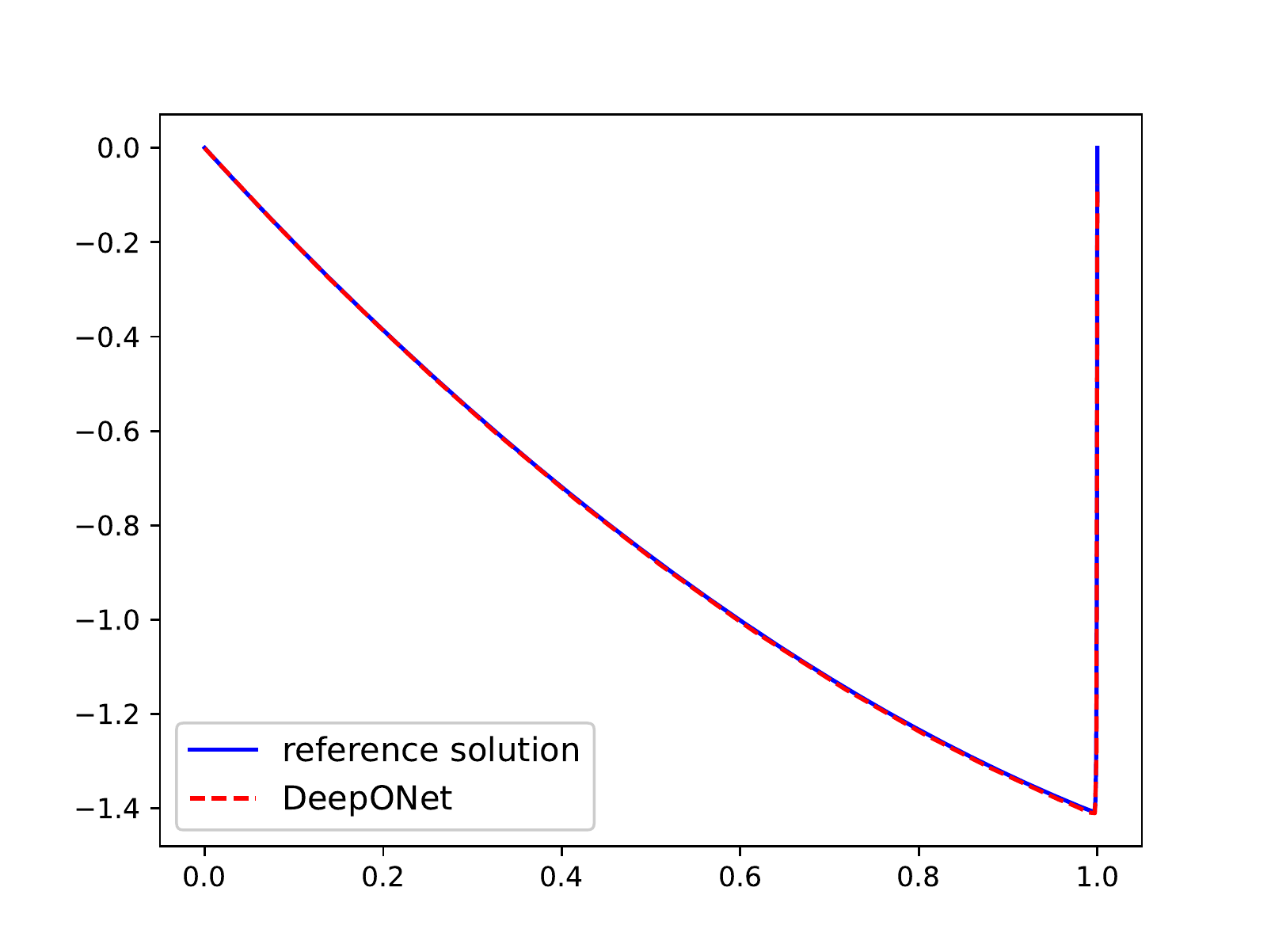}\\
\scriptsize{a)}
\end{minipage}
\begin{minipage}{0.325\textwidth}
\centering
\begin{minipage}{\textwidth}
\centering
\includegraphics[width=\textwidth]{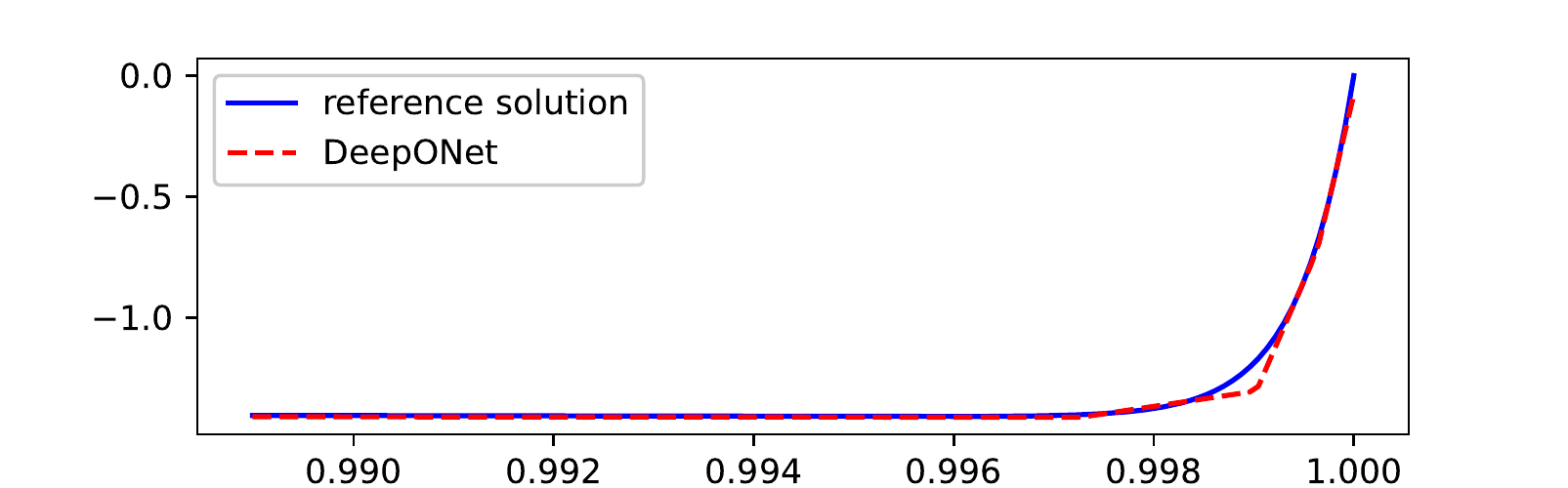}\\
\scriptsize{b)}
\end{minipage}
\\
\begin{minipage}{\textwidth}
\centering
\includegraphics[width=\textwidth]{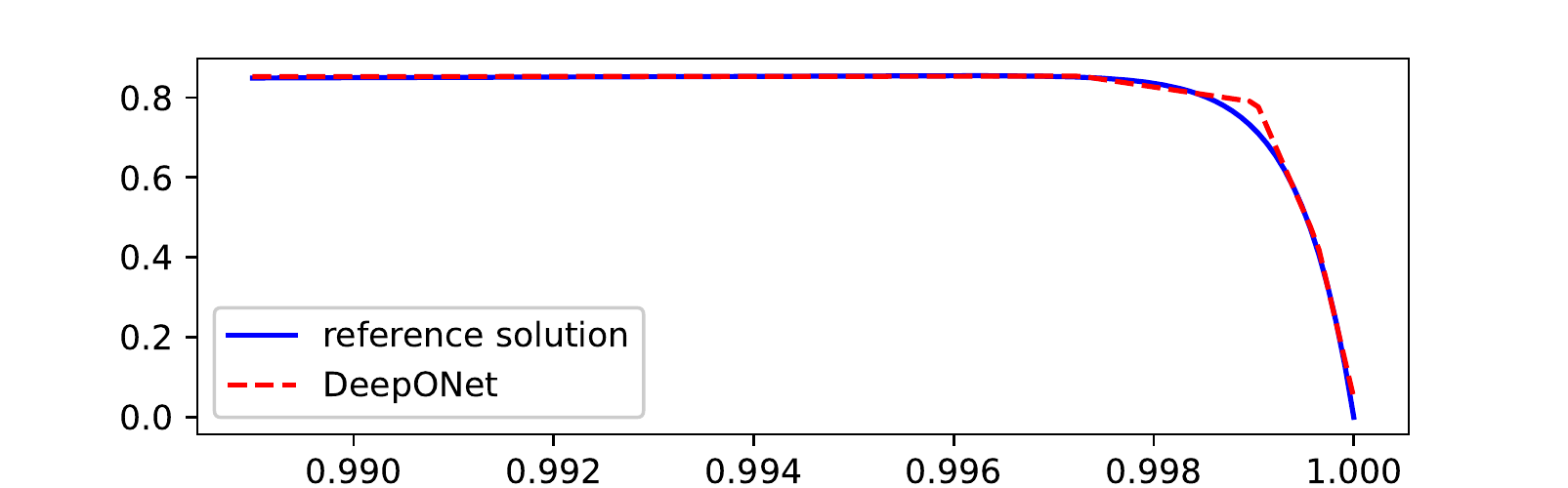}\\
\scriptsize{c)}
\end{minipage}
\end{minipage}
\begin{minipage}{0.325\textwidth}
\centering
\includegraphics[width=\textwidth]{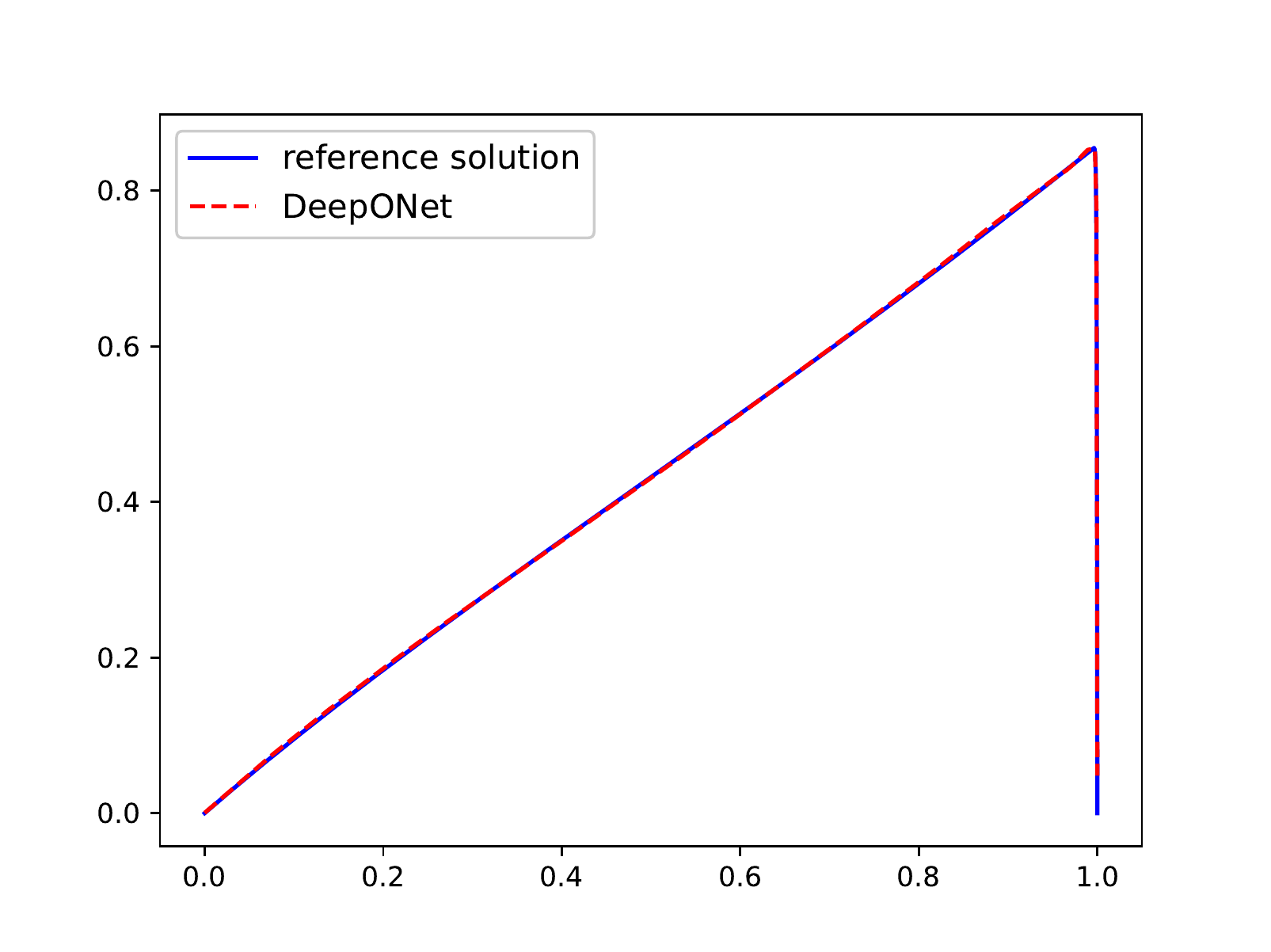}\\
\scriptsize{d)}
\end{minipage}
  \caption{The trained model's predictions and reference solutions for Eq.~\eqref{exam2} ($\varepsilon=0.001$) on Shishkin mesh. a): Prediction and reference solution on a random $f\sim\mu$ on a grid of 256 points, with the latter using an up-winding scheme on Shishkin mesh. b) Horizontal magnification of the region $[1-\sigma,1]$ in a). c) Horizontal magnification of the region $[1-\sigma,1]$ in d). d): the model's prediction and reference solution for $f(x)=e^x$ on a grid of 256 points.}
  \label{fig4}
\end{figure}

Based on Theorem \ref{empirical risk thm}, it can be observed that the empirical loss showcases a negative correlation with the quantity of $f$ samples as well as the number of Shishkin locations $y$. Importantly, this relationship holds irrespective of $1/\varepsilon$. Subsequently, we endeavor to establish the validity of this relationship for Eq.~\eqref{exam2}. The results obtained from pertinent experiments are illustrated in Fig.~\ref{fig5}.

\begin{figure}[!tbh]
\centering
\begin{minipage}{0.4\textwidth}
\centering
\includegraphics[width=\textwidth]{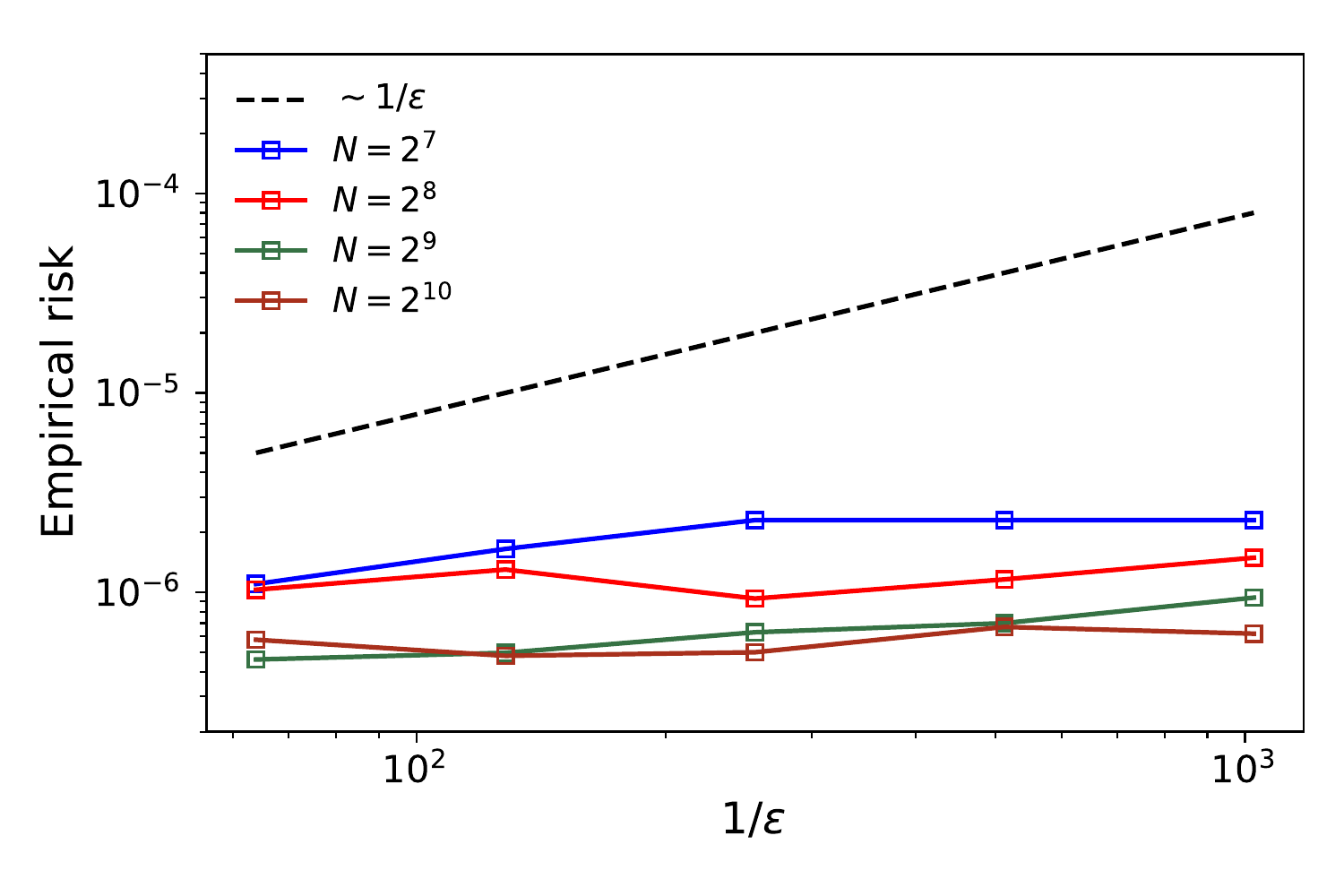}\\
\scriptsize{a)}
\end{minipage}
\begin{minipage}{0.4\textwidth}
\centering
\includegraphics[width=\textwidth]{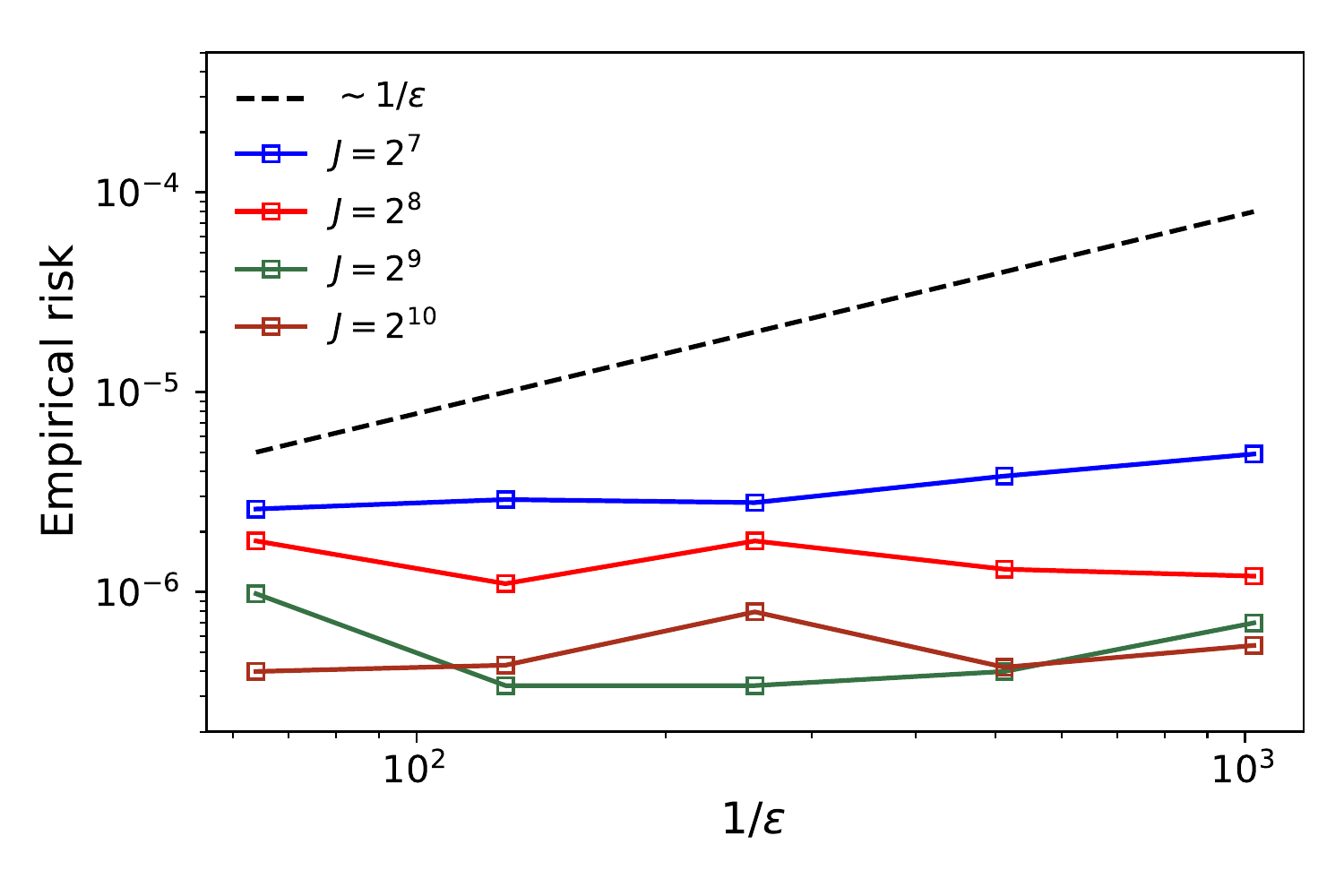}\\
\scriptsize{b)}
\end{minipage}
  \caption{The generalization ability of trained models with varying factors. a): Impact of varying $f$ sample numbers on model generalization with a fixed set of 512 Shishkin locations. Empirical risk is calculated over a sample of $2^{21}$ pairs $(f, y)$ consisting of 4096 randomly selected $f$ samples and 512 locations $y$. b): Impact of different numbers of locations $y$ on model generalization with a fixed set of 1000 $f$ samples. Empirical risk was calculated over $1000\times2^{12}$ sample pairs $(f, y)$.}
 \label{fig5}
\end{figure}
\end{example}

Fig.\ref{fig1} and Fig.\ref{fig4} illustrate DeepONets' effectiveness in capturing the boundary layer behavior near $x=1$ and performing remarkably accurate out-of-distribution predictions. These experimental findings suggest that DeepONets can serve as a reliable approximator for approximating the solution operator of one-dimensional singular perturbation problems. As we extend our focus to the two-dimensional problem, it is important to note that theoretical results currently exist only for the one-dimensional problem. However, we expect that DeepONet is also suitable for solving the two-dimensional singular perturbation problem, and subsequent numerical results validate this conjecture.
\begin{example}
Consider a 2D singularly perturbed problem:
    \begin{equation}
        \left\{ \begin{array}{l}
-\varepsilon\Delta u+u_x+u_y+u=f,\quad x\in\Omega=(0,1)\times(0,1),\\
u|_{\partial\Omega}=0.
\end{array} \right.\label{exam3}
    \end{equation}

In this study, we employ DeepONets to learn the solution operator for this problem with $\varepsilon=0.001$, that is, to map $f$ to the solution $u$ of Eq.~\eqref{exam3}. As the dimensionality of the problem increases, the size of the training data set needs to be significantly increased due to the heightened demand for locations. To facilitate better learning of the equation information by the operator network, penalization of certain locations may be necessary, especially when the network cannot be trained on enough data. In our study, we chose to penalize the boundary points in the loss function, giving rise to a loss function of the following form:
\begin{equation*}
    \mathcal{L}(\boldsymbol{\theta})=\frac{1}{NJ_r}\sum\limits_{n=1}^{N}\sum\limits_{j=1}^{J_r}|u_n(\boldsymbol{y}_j)-\mathscr{N}_{\boldsymbol{\theta}}(F_n)(\boldsymbol{y}_j)| ^2+\frac{\lambda}{NJ_b}\sum\limits_{n=1}^{N}\sum\limits_{j=1}^{J_b}|u_n(\boldsymbol{y}_j)-\mathscr{N}_{\boldsymbol{\theta}}(F_n)(\boldsymbol{y}_j)| ^2,
\end{equation*}
where $\lambda$ is the positive penalizing parameter.
\begin{figure}[!tbh]
\centering
\begin{minipage}{0.23\textwidth}
\centering
\includegraphics[width=\textwidth]{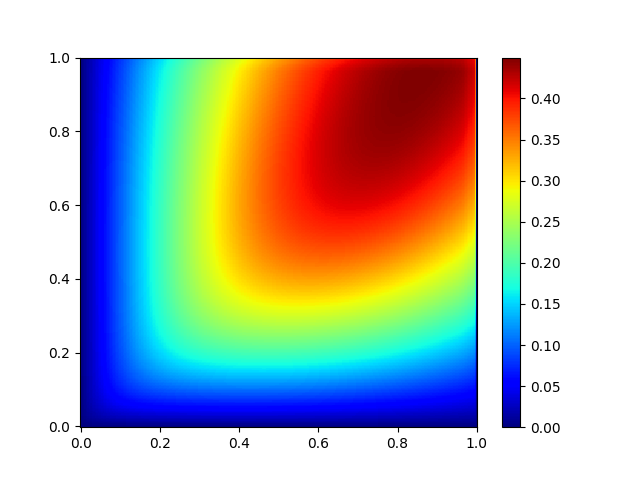}\\
\scriptsize{a)}
\end{minipage}
\begin{minipage}{0.23\textwidth}
\centering
\begin{minipage}{\textwidth}
\centering
\includegraphics[width=\textwidth]{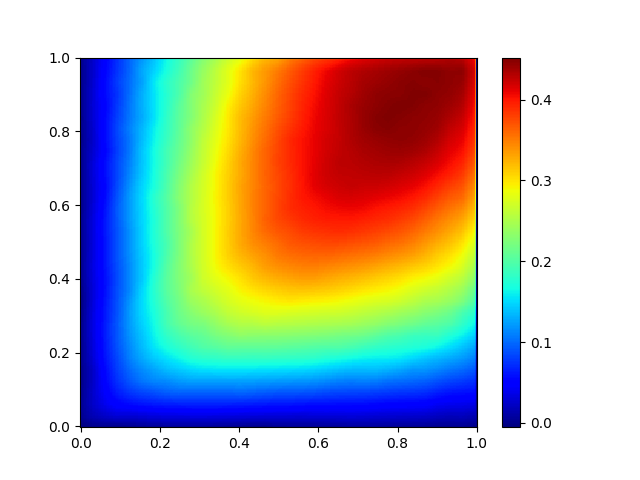}\\
\end{minipage}
\\
\begin{minipage}{\textwidth}
\centering
\includegraphics[width=\textwidth]{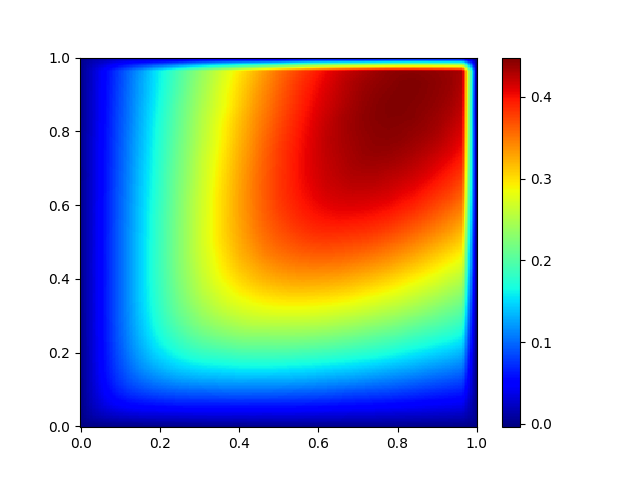}\\
\scriptsize{b)}
\end{minipage}
\end{minipage}
\begin{minipage}{0.23\textwidth}
\centering
\begin{minipage}{\textwidth}
\centering
\includegraphics[width=\textwidth]{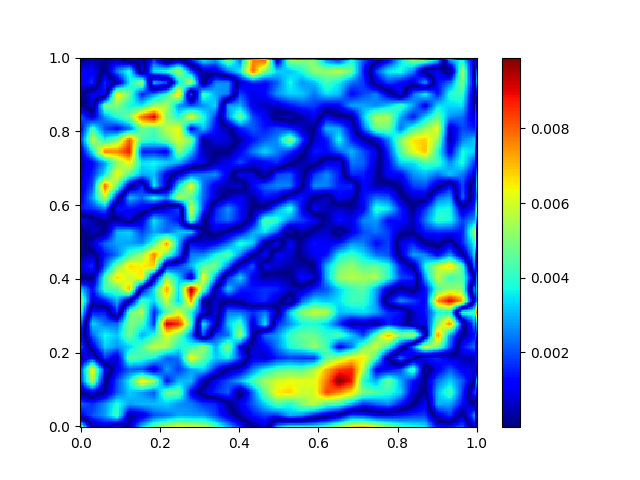}\\
\end{minipage}
\\
\begin{minipage}{\textwidth}
\centering
\includegraphics[width=\textwidth]{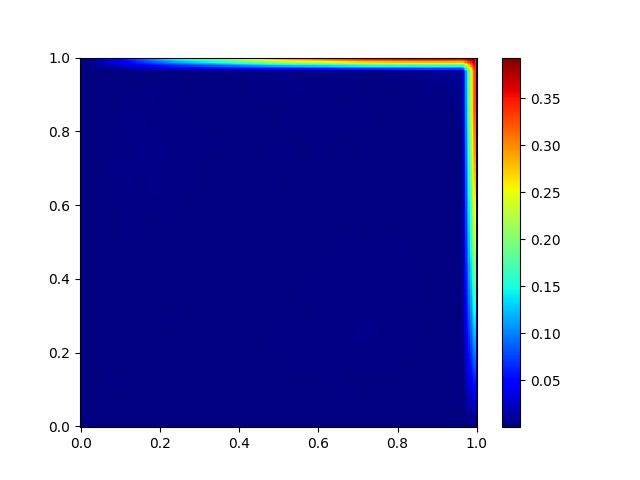}\\
\scriptsize{c)}
\end{minipage}
\end{minipage}
\begin{minipage}{0.23\textwidth}
\centering
\begin{minipage}{\textwidth}
\centering
\includegraphics[width=\textwidth]{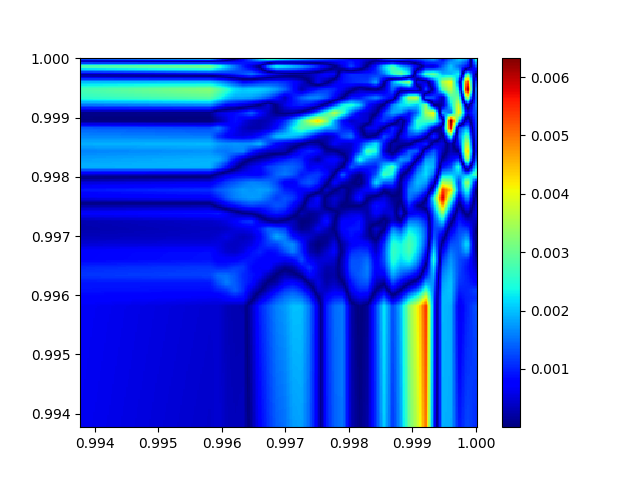}\\
\end{minipage}
\\
\begin{minipage}{\textwidth}
\centering
\includegraphics[width=\textwidth]{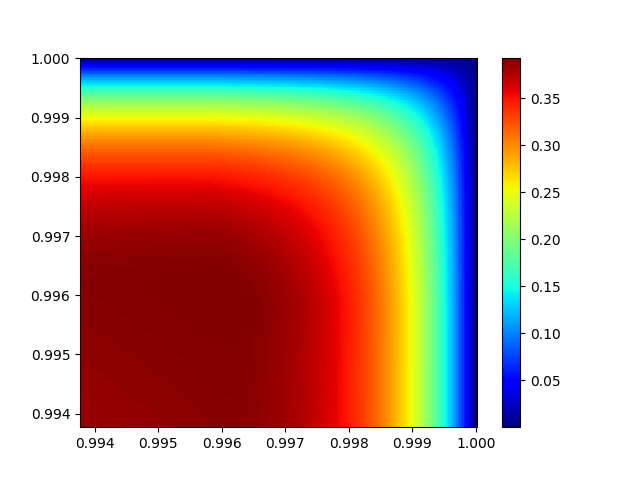}\\
\scriptsize{d)}
\end{minipage}
\end{minipage}
\caption{Comparison of performance between models trained on Shishkin and equidistant locations. a): The reference solution for a Gaussian random field $f\sim\mathcal{G}(0,k_{RBF})$ achieved through an up-winding scheme on the Shishkin mesh. b): The predictions of the trained models on $f$, where the upper subfigure displays the prediction of the model trained on the Shishkin positions and the lower subfigure shows the prediction of the model trained on the equidistant positions. c): The errors between the predicted results of the trained models and the reference solution. d): An enlargement of the $[1-\sigma,1]\times[1-\sigma,1]$ region in the corresponding subfigures of c).}
  \label{fig6}
\end{figure}

Subsequently, we proceed to train DeepONets utilizing a dataset of $1000\times64\times64$ triples $(f, \boldsymbol{y}, u)$, consisting of 1000 Gaussian random fields $f\sim\mathcal{G}(0,k_{RBF})$ and $64\times64$ locations $\boldsymbol{y}$, with $k_{RBF}$ denoting the radial basis function kernel. The training process is conducted over 2000 epochs, during which the penalizing parameter is set to $\lambda=0.1$. We conduct training on both equidistant and Shishkin locations, with a comparison of their performance shown in Fig.~\ref{fig6}. Our findings indicate that the network trained on equidistant locations effectively captures the behavior of the solution outside the boundary layer, but falls short in characterizing the behavior inside the boundary layer. In contrast, the network trained on Shishkin locations proves adept at capturing the behavior of the solution within the boundary layer.
\end{example}
\section{Conclusion and Future Directions\label{sec_conclusion}}
In this manuscript, we present a novel application of DeepONets for solving singular perturbation problems, specifically the one-dimensional convection-diffusion equation. We conduct a thorough analysis of the approximation error generated by the operator network when approximating the target solution operator and provide proof of its convergence rate with respect to the input and output dimensions of the branch net. Importantly, we demonstrate that this error is independent of $\varepsilon$, which is a notable finding. Since this error cannot be calculated directly, the empirical risk serves as a surrogate. Notably, the Shishkin mesh points are employed as samples on the interval $[0,1]$. We analyze the empirical risk and the corresponding generalization gap, both of which exhibit $\varepsilon$-uniform convergence with respect to the number of samples. Additionally, the Shishkin mesh points are utilized as locations for the loss function, and numerical experiments showcase the effectiveness of DeepONets in capturing the boundary layer behavior of solutions to singular perturbation problems. Our findings provide a novel approach to addressing this type of problem.

We investigate the steady-state singular perturbation problems in one dimension. However, we expect that our methodology could also be applied to problems in higher dimensions or those that involve time dependence. Despite their effectiveness as purely data-driven models, DeepONets fail to fully exploit the information encapsulated within the governing equation, leading to inefficiencies. A promising avenue for further development is to incorporate singular perturbation theory, as demonstrated in previous works such as \cite{BLpinn}. Such improvements have the potential to enhance the performance of traditional models significantly.

\section*{Acknowledgments}
Z.Y. Huang was partially supported by NSFC Projects No.~ 12025104, 81930119. Y. Li was partially supported by NSFC Projects No.~ 62106103 and Fundamental Research Funds for the Central Universities No.~ ILA22023.


\section{Appendix\label{sec_appendix}}
\subsection{Proof of Lemma \ref{expect_interpolation_error}\label{appendix 7.1}}
\begin{proof}
By utilizing the condition, $\mathbb{E}_{f\sim\mu}|\frac{d}{dx}F(f)(x)|\le C(1+1/\varepsilon e^{-\alpha(1-x)/\varepsilon})$, we derive
\begin{equation*}
    \begin{aligned}
    \mathbb{E}_{f\sim\mu}\left|\int_0^1F(f)(x)dx-\sum\limits_{j=1}^{J}h_{j-1}F(f)(x_j)\right|&=\mathbb{E}_{f\sim\mu}\left|\sum\limits_{j=1}^{J}\int_{x_{j-1}}^{x_j}F(f)(x)-F(f)(x_j)dx\right|\\
    &=\mathbb{E}_{f\sim\mu}\left|\sum\limits_{j=1}^{J}\int_{x_{j-1}}^{x_j}\int_x^{x_j}\frac{d}{ds}F(f)(s)dsdx\right|\\
    &\le\sum\limits_{j=1}^{J}\int_{x_{j-1}}^{x_j}\int_x^{x_j}\mathbb{E}_{f\sim\mu}\left|\frac{d}{ds}F(f)(s)\right|dsdx\\
    &\le C\sum\limits_{j=1}^{J}\int_{x_{j-1}}^{x_j}\int_x^{x_j}(1+\frac{1}{\varepsilon} e^{-\alpha(1-s)/\varepsilon})dsdx.
    \end{aligned}
\end{equation*}

We consider two cases for specifying $\sigma$, based on its definition. when $\sigma=2\varepsilon\ln{J}/\alpha$, namely $2\varepsilon\ln{J}/\alpha\le1/2$, we obtain the following result for $x_j \leq 1 - \sigma$:
\begin{equation*}
\begin{aligned}
    \int_{x_{j-1}}^{x_j} \int_x^{x_j} (1+\frac{1}{\varepsilon} e^{-\alpha(1-s)/\varepsilon}) dsdx =& \frac{h_{j-1}^2}{2}-\frac{\varepsilon}{\alpha^2}(e^{-\alpha(1-x_j)/\varepsilon}-e^{-\alpha(1-x_{j-1})/\varepsilon})+\frac{h_{j-1}}{\alpha}e^{-\alpha(1-x_j)/\varepsilon} \\
    \le& \frac{h_{j-1}^2}{2}+\frac{h_{j-1}}{\alpha}e^{-\alpha(1-x_j)/\varepsilon} \\
    \le & \frac{h_{j-1}^2}{2}+\frac{h_{j-1}}{\alpha J^2}\\
    \le& 2(1+1/\alpha)\frac{1}{J^2}.
\end{aligned}
\end{equation*}
For $x_{j-1} >1-\sigma$, we have
\begin{equation*}
\int_{x_{j-1}}^{x_j}  \int_x^{x_j} (1+\frac{1}{\varepsilon} e^{-\alpha(1-s)/\varepsilon}) dsdx\le \frac{2}{\varepsilon} \int_{x_{j-1}}^{x_j}  \int_x^{x_j} dsdx= \frac{h_{j-1}^2}{\varepsilon}\le \frac{4\ln{J}}{\alpha J^2},
\end{equation*}
which implies that
\begin{equation*}
    C\sum\limits_{j=1}^{J}\int_{x_{j-1}}^{x_j}\int_x^{x_j}(1+\frac{1}{\varepsilon} e^{-\alpha(1-s)/\varepsilon})dsdx \le 2C(1+\frac{2}{\alpha})\frac{\ln J}{J}.
\end{equation*}

The case $\sigma = \frac{1}{2}$ is straightforward, as we directly obtain:
\begin{equation*}
C\sum\limits_{j=1}^{J}\int_{x_{j-1}}^{x_j}\int_x^{x_j}(1+\frac{1}{\varepsilon} e^{-\alpha(1-s)/\varepsilon})dsdx\le \frac{2C}{\varepsilon}\sum\limits_{j=1}^{J}\int_{x_{j-1}}^{x_j}\int_x^{x_j} dsdx = \frac{C}{\varepsilon J}<2C(1+\frac{2}{\alpha})\frac{\ln J}{J}.
\end{equation*}
Then the Lemma is proved.
\qed
\end{proof}
\subsection{Proof of Lemma \ref{lem_diff124}\label{appendix 7.2}}
    \begin{proof}
        To facilitate our proof procedure, we first establish a lemma that is essential to our subsequent analysis.
    \begin{lemma}
    [cf.~Stuart~\cite{stuart_inverse}]If $\mathcal{N}(0,\mathcal{C})$ is a Gaussian measure on a Hilbert space  $\mathcal{H}$, then for any integer $q$, there is a constant $C=C_q\ge 0$ such that, for $x\sim\mathcal{N}(0,\mathcal{C})$, 
    \begin{equation*}
        \mathbb{E}\|x\|^{2q}\le C_q(\mathbb{E}\|x\|^2)^q.
    \end{equation*}
    \label{g.m_inequ}
\end{lemma}

Let $X_n = \|\mathscr{G}(F_n) - \hspace{0.5em}\overline{\kern-0.5em\mathscr{N}\kern+0.1em}\hspace{-0.1em}(F_n)\|_{L^2}^2$, and define $S^N = \frac{1}{N}\sum\limits_{n=1}^N X_n$. Using basic inequalities such as the H$\ddot{\text{o}}$lder inequality, we obtain:
    \begin{equation}
        \uppercase\expandafter{\romannumeral1}=\mathbb{E}\left|\sqrt{S^N}-\sqrt{\mathbb{E}S^N}\right|\le \left(\mathbb{E}|S^N-\mathbb{E}(S^N)|^2\right)^{1/4}=\frac{1}{N^{1/4}}\left(\mathbb{E}(X_1-\mathbb{E}X_1)^2\right)^{1/4}.\label{diff1}
    \end{equation}
Here, $\mathbb{E}X_1$ is closely related to the approximation error on $\hspace{0.5em}\overline{\kern-0.5em\mathscr{N}\kern+0.1em}\hspace{-0.1em}$, specifically $\mathbb{E}X_1 = \widehat{\mathscr{E}}(\hspace{0.5em}\overline{\kern-0.5em\mathscr{N}\kern+0.1em}\hspace{-0.1em})^2$. For $\mathbb{E}X_1^2$, we have:
\begin{equation*}
    \mathbb{E}X_1^2=\int_{L^2}\|\mathscr{G}(f)-\hspace{0.5em}\overline{\kern-0.5em\mathscr{N}\kern+0.1em}\hspace{-0.1em}(f)\|_{L^2}^4d\mu(f)
        =\int_{L^2}\|f\|_{L^2}^4d(\mathscr{G}-\hspace{0.5em}\overline{\kern-0.5em\mathscr{N}\kern+0.1em}\hspace{-0.1em})_{\#}\mu(f),
\end{equation*}
where $(\mathscr{G} - \hspace{0.5em}\overline{\kern-0.5em\mathscr{N}\kern+0.1em}\hspace{-0.1em})_{\#}\mu$ is the push-forward measure of $\mu$ under $\mathscr{G}-\hspace{0.5em}\overline{\kern-0.5em\mathscr{N}\kern+0.1em}\hspace{-0.1em}$. As $\mathscr{G}$ and $\hspace{0.5em}\overline{\kern-0.5em\mathscr{N}\kern+0.1em}\hspace{-0.1em}$ are bounded linear operators, $(\mathscr{G} - \hspace{0.5em}\overline{\kern-0.5em\mathscr{N}\kern+0.1em}\hspace{-0.1em})_{\#}\mu$ is a Gaussian measure. By applying Lemma \ref{g.m_inequ} to the equation above and combining it with \eqref{diff1}, we can derive the following result:
\begin{equation*}
\uppercase\expandafter{\romannumeral1}\lesssim\widehat{\mathscr{E}}(\hspace{0.5em}\overline{\kern-0.5em\mathscr{N}\kern+0.1em}\hspace{-0.1em})/N^{1/4}.
\end{equation*}

  Next, we can bound the term \uppercase\expandafter{\romannumeral2} as follows:
        \begin{equation*}
        \begin{aligned}
\uppercase\expandafter{\romannumeral2}\le&\mathbb{E}\left(\frac{1}{N}\sum\limits_{n=1}^N\|\hspace{0.5em}\overline{\kern-0.5em\mathscr{N}\kern+0.1em}\hspace{-0.1em}(F_n)-\widetilde{\mathscr{N}}(F_n)\|_{L^2}^2\right)^{1/2}\\
        =&\mathbb{E}\left(\frac{1}{N}\sum\limits_{n=1}^N\left\|\sum\limits_{q=-K}^{K}(\mathscr{G}\circ\mathscr{D}\circ\mathscr{E}(F_n),\widetilde{\tau}_q)(\tau_q-\widetilde{\tau}_q)\right\|_{L^2}^2\right)^{1/2}\\
        \le&\eta\mathbb{E}\left(\frac{1}{N}\sum\limits_{n=1}^N\|\mathscr{G}\circ\mathscr{D}\circ\mathscr{E}(F_n)\|_{L^2}^2\right)^{1/2}\\
        \le&\eta Lip(\mathscr{G})\left(\mathbb{E}\|\mathscr{D}\circ\mathscr{E}(F_1)\|_{L^2}^2\right)^{1/2}.
        \end{aligned}
    \end{equation*}
    Note that  $\mathscr{D}\circ\mathscr{E}(e^{2\pi ikx})=e^{2\pi ikx}\ (|k|\le M)$, where $P_M$ is the projection operator onto $span\{e^{2\pi ikx}:|k|\le M\}$. Thus, we have $\mathscr{D}\circ\mathscr{E}\circ P_M=P_M$, and since $\mathscr{D}\circ\mathscr{E}$ is linear, we get
\begin{equation*}
\mathscr{D}\circ\mathscr{E}=\mathscr{D}\circ\mathscr{E}\circ P_M+\mathscr{D}\circ\mathscr{E}\circ P_M^{\perp}=P_M+\mathscr{D}\circ\mathscr{E}\circ P_M^{\perp}.
\end{equation*}
Hence,
        \begin{equation*}
        \begin{aligned}
       \left(\mathbb{E}\|\mathscr{D}\circ\mathscr{E}(F_1)\|_{L^2}^2\right)^{1/2}\le&\left(\mathbb{E}\|P_MF_1\|_{L^2}^2\right)^{1/2}+\left(\mathbb{E}\|\mathscr{D}\circ\mathscr{E}\circ (P_M^{\perp}F_1)\|_{L^2}^2\right)^{1/2}\\
       \le&\left(\mathbb{E}\|F_1\|_{L^2}^2\right)^{1/2}+\sqrt{\sum\limits_{|k|>M}\lambda_k}\\
       =&\sqrt{\sum\limits_{k\in\mathbb{Z}}\lambda_k}+\sqrt{\sum\limits_{|k|>M}\lambda_k}\\
       \le&2\sqrt{\sum\limits_{k\in\mathbb{Z}}\lambda_k}\le 2.
        \end{aligned}
    \end{equation*}
By combining the above results, we obtain 
\begin{equation*}
    \uppercase\expandafter{\romannumeral2}\lesssim \eta =p^{-n}.
\end{equation*}

Similarly, the estimate of \uppercase\expandafter{\romannumeral4} follows the same procedure as for \uppercase\expandafter{\romannumeral2}, yielding 
\begin{equation*}
    \uppercase\expandafter{\romannumeral4}\lesssim \eta =p^{-n}.
\end{equation*}
By combining the above estimates for \uppercase\expandafter{\romannumeral1}, \uppercase\expandafter{\romannumeral2}, and \uppercase\expandafter{\romannumeral4}, we arrive at the final result.

    \qed
    \end{proof}
\subsection{Proof of Lemma \ref{deri_bound}\label{appendix 7.3}}
The proof of this theorem relies on a series of lemmas, which we will now present along with their corresponding proofs.
\begin{lemma}
Let $f$ drawn from the measure $\mu$, and let $\widehat{f}_k$ and $c_k$ denote the Fourier coefficients and discrete Fourier coefficients of $f$, respectively. Specifically, we have, 
\begin{equation*}
    \widehat{f}_k=\int_0^1 f(x)e^{-i2\pi kx}dx,\ c_k=\frac{1}{m}\sum\limits_{j=1}^mf(x_j)e^{-\frac{i2\pi jk}{m}}.
\end{equation*}
For $0\le k\le M$, it can be shown that
\begin{equation*}
   (\mathbb{E}_{f\sim\mu}|\widehat{f}_k|^2)^{1/2}\le Ce^{-\pi^2l^2k^2}, 
\end{equation*}
\begin{equation*}
    (\mathbb{E}_{f\sim\mu}|c_k|^2)^{1/2}\le Ce^{-\pi^2l^2k^2},
\end{equation*}
\begin{equation*}
    \left(\mathbb{E}_{f\sim\mu}|\widehat{f}_k-c_k|^2\right)^{1/2}\le \frac{C}{m}e^{-\pi^2l^2k^2},
\end{equation*}
    where $C$ is a constant that solely depends on $l$.\label{lem_fourier coff}
\end{lemma}
\begin{proof}
By the K-L expansion, we have the representation of $f\sim\mu$ as
\begin{equation*}
    f(x)=\sum\limits_{n=-\infty}^{\infty}\sqrt{\lambda_n}\xi_n\phi_n(x),
\end{equation*}
where $\lambda_n=\sqrt{2\pi}l e^{-2\pi^2n^2l^2}$, $\phi_n(x)=e^{i2\pi nx}$ and $\xi_n\sim\mathcal{N}(0,1)$ are i.i.d. Gaussian random variables. Then
\begin{equation*}
   (\mathbb{E}_{f\sim\mu}|\widehat{f}_k|^2)^{1/2}= \left(\mathbb{E}|\sqrt{\lambda_k}\xi_k|^2\right)^{1/2}=\sqrt{\lambda_k}=(\sqrt{2\pi}l)^{1/2}e^{-\pi^2k^2l^2}. 
\end{equation*}
On the other hand, we can obtain
\begin{equation*}
    \begin{aligned}
        \left(\mathbb{E}_{f\sim\mu}|\widehat{f}_k-c_k|^2\right)^{1/2}&=\left(\mathbb{E}_{f\sim\mu}\left|\widehat{f}_k-\frac{1}{m}\sum\limits_{j=1}^mf(x_j)e^{-\frac{i2\pi jk}{m}}\right|^2\right)^{1/2}\\
        &=\left(\mathbb{E}_{f\sim\mu}\left|\widehat{f}_k-\frac{1}{m}\sum\limits_{j=1}^m\sum\limits_{n=-\infty}^{\infty}\widehat{f}_ne^{\frac{i2\pi jn}{m}}e^{-\frac{i2\pi jk}{m}}\right|^2\right)^{1/2}\\
        &=\left(\mathbb{E}_{f\sim\mu}\left|\widehat{f}_k-\sum\limits_{n=-\infty}^{\infty}\widehat{f}_{k+nm}\right|^2\right)^{1/2}\\
        &\lesssim l^{1/2}\sum\limits_{n=1}^{\infty}e^{-\pi^2l^2(k+nm)^2}+e^{-\pi^2l^2(k-nm)^2},
        \end{aligned}
\end{equation*}
where
\begin{equation*}
    \sum\limits_{n=1}^{\infty}e^{-\pi^2l^2(k+nm)^2}\le \int_0^{+\infty}e^{-\pi^2l^2(k+xm)^2}dx=\frac{1}{\pi lm}\int_{k\pi l}^{+\infty}e^{-y^2}dy\lesssim\frac{1}{ml}e^{-\pi^2l^2k^2},
\end{equation*}
and considering that $0\le k\le M$, we have
\begin{equation*}
    \sum\limits_{n=1}^{\infty}e^{-\pi^2l^2(k-nm)^2}\le e^{-\pi^2l^2k^2}\sum\limits_{n=1}^{\infty}e^{-\pi^2l^2nm}\le e^{-\pi^2l^2k^2}\int_0^{+\infty}e^{-\pi^2l^2xm}dx=\frac{1}{\pi^2l^2m}e^{-\pi^2l^2k^2},
\end{equation*}
thus
\begin{equation*}
    \left(\mathbb{E}_{f\sim\mu}|\widehat{f}_k-c_k|^2\right)^{1/2}\lesssim \frac{1}{m}(l^{-1/2}+l^{-3/2})e^{-\pi^2l^2k^2}.
\end{equation*}
    
    Furthermore, $c_k$ can be bounded by:
\begin{equation*}
    (\mathbb{E}_{f\sim\mu}|c_k|^2)^{1/2}\le\left(\mathbb{E}_{f\sim\mu}|\widehat{f}_k-c_k|^2\right)^{1/2}+(\mathbb{E}_{f\sim\mu}|\widehat{f}_k|^2)^{1/2}\le Ce^{-\pi^2l^2k^2},
\end{equation*}
where $C$ is a constant that depends only on $l$.
    \qed
\end{proof}

Directly applying Lemma \ref{lem_fourier coff}, we obtain the following corollary for $k\le M$:
\begin{corollary}
\begin{equation*}
    \left(\mathbb{E}_{f\sim\mu}\left|f-\sum\limits_{q=-k}^kc_q\phi_q(x)\right|^2\right)^{1/2}\le \frac{C}{k},\ \left(\mathbb{E}_{f\sim\mu}\left|\mathscr{G}\left(f-\sum\limits_{q=-k}^kc_q\phi_q(x)\right)\right|^2\right)^{1/2}\le \frac{C}{k},
\end{equation*}
where $C$ depends solely on $l$.
\label{differ_norm}
\end{corollary}
\begin{proof}
To establish the first inequality, we proceed as follows:
    \begin{equation*}
        \begin{aligned}
            \left(\mathbb{E}_{f\sim\mu}\left|f-\sum\limits_{q=-k}^kc_q\phi_q(x)\right|^2\right)^{1/2}&= \left(\mathbb{E}_{f\sim\mu}\left|\sum_{q=-\infty}^{\infty} \widehat{f}_q \phi_q - \sum_{|q|\le k} c_q \phi_q\right|^2\right)^{1/2} \\
            &=\left( \mathbb{E}_{f\sim\mu}\left|\sum_{|q|\le k} (\widehat{f}_q - c_q) \phi_q + \sum_{|q|> k} \widehat{f}_q \phi_q\right|^2\right)^{1/2} \\
            &\le \sum_{|q|\le k}\left(\mathbb{E}_{f\sim\mu}| \widehat{f}_q - c_q|^2\right)^{1/2}  + \sum_{|q|> k} \left(\mathbb{E}_{f\sim\mu}| \widehat{f}_q|^2\right)^{1/2} \\
            &\le \frac{C}{m}\sum_{|q|\le k} e^{-\pi^2 l^2 q^2} + C\sum_{|q|> k}e^{-\pi^2 l^2 q^2} \\
            &\le\frac{C}{k}.
        \end{aligned}
    \end{equation*}
    Based on Lemma \ref{lem_bounded}, it follows that $|\mathscr{G}(\phi_k)|\lesssim 1$. Moreover, by repeating the aforementioned proof, we can derive the second inequality stated in this corollary.
    \qed
\end{proof}

We present our final lemma, the proof of which closely follows that of \cite{kellogg_ssp}. 
\begin{lemma}
Let $u=\mathscr{G}(f)$ be the solution to the equation
\begin{equation*}
    \left\{\begin{array}{l}
        -\varepsilon u''+p(x)u'+q(x)u=f,\quad x\in (0,1),\\
        u(0)=u(1)=0,
    \end{array}\right.
\end{equation*}
and let $h=f-qu$. Then, we have the estimate:
\begin{equation*}
    |u'(1)|\lesssim\frac{1}{\varepsilon^2}\int_0^1\int_t^1e^{-\alpha(s-t)/\varepsilon}|h(s)|dsdt.
\end{equation*}
    \label{u'_estimate}
\end{lemma}
\begin{proof}
    The function $u$ satisfies
    \begin{equation}
        u(x)=u_p(x)+K_1+K_2\int_x^1e^{-\frac{1}{\varepsilon}\int_t^1p(s)ds}dt,\label{equ_u}
    \end{equation}
    where $K_1=u(1)$, $K_2=u'(1)$, and $u_p(x)=-\int_x^1\int_t^1\frac{1}{\varepsilon}e^{-\frac{1}{\varepsilon}\int_t^sp(w)dw}h(s)dsdt$. By setting $x=0$ in \eqref{equ_u}, we obtain the following result:
    \begin{equation*}
        u(0)=u_p(0)+K_1+K_2\int_0^1e^{-\frac{1}{\varepsilon}\int_t^1p(s)ds}dt.
    \end{equation*}
    Given that $u(0)=u(1)=0$ and $p(x)\ge\alpha>0$ is bounded on $[0,1]$, it follows that:
    \begin{equation*}
        \varepsilon|K_2|\lesssim |K_2|\int_0^1e^{-c(1-t)/\varepsilon}dt\le|u_p(0)|\le \int_0^1\int_t^1\frac{1}{\varepsilon}e^{-\alpha(s-t)/\varepsilon}|h(s)|dsdt,
    \end{equation*}
    which implies that
    \begin{equation*}
        |u'(1)|=|K_2|\lesssim\frac{1}{\varepsilon^2}\int_0^1\int_t^1e^{-\alpha(s-t)/\varepsilon}|h(s)|dsdt.
    \end{equation*}
    \qed
\end{proof}

Using the lemmas established above, we can now provide the proof of Lemma \ref{deri_bound}.

\begin{proof}
    Let $u(x)=\mathscr{G}(f-\mathscr{D}\circ\mathscr{E}(f))(x)$ and $h(x)=f-\mathscr{D}\circ\mathscr{E}(f)-q(x)u$, then $z(x)=du/dx$ satisfies
    \begin{equation*}
        z(x)=z(1)e^{-\frac{1}{\varepsilon}\int_x^1p(s)ds}+\frac{1}{\varepsilon}\int_x^1h(t)e^{-\frac{1}{\varepsilon}\int_x^tp(s)ds}dt.
    \end{equation*}
    Employing Eq.~\eqref{simplify}, Lemma \ref{u'_estimate}, and the fact that $p(x)\ge \alpha>0$, yields the following result:
\begin{align}
&\mathbb{E}_{f\sim\mu}\left|\left(\mathscr{G}(f)(x)-\widetilde{\mathscr{N}}(f)(x)\right)\frac{d}{dx}\left(\mathscr{G}(f)(x)-\widetilde{\mathscr{N}}(f)(x)\right)\right|\nonumber\\
=&\mathbb{E}_{f\sim\mu}\left|\mathscr{G}(f-\mathscr{D}\circ\mathscr{E}(f))(x)\left[z(1)e^{-\frac{1}{\varepsilon}\int_x^1p(s)ds}+\frac{1}{\varepsilon}\int_x^1h(t)e^{-\frac{1}{\varepsilon}\int_x^tp(s)ds}dt\right]\right|\nonumber\\
\le& e^{-\frac{1}{\varepsilon}\int_x^1p(s)ds}\mathbb{E}_{f\sim\mu}\left|\mathscr{G}(f-\mathscr{D}\circ\mathscr{E}(f))(x)z(1)\right|+\frac{1}{\varepsilon}\int_x^1e^{-\frac{1}{\varepsilon}\int_x^tp(s)ds}\mathbb{E}_{f\sim\mu}\left|\mathscr{G}(f-\mathscr{D}\circ\mathscr{E}(f))(x)h(t)\right|dt\nonumber\\
\lesssim&\frac{1}{\varepsilon^2}e^{-\alpha(1-x)/\varepsilon}\int_0^1\int_t^1e^{-\alpha(s-t)/\varepsilon}\mathbb{E}_{f\sim\mu}\left|\mathscr{G}(f-\mathscr{D}\circ\mathscr{E}(f))(x)h(s)\right|dsdt\nonumber\\
&+\frac{1}{\varepsilon}\int_x^1e^{-\alpha(t-x)/\varepsilon}\mathbb{E}_{f\sim\mu}\left|\mathscr{G}(f-\mathscr{D}\circ\mathscr{E}(f))(x)h(t)\right|dt.\label{noname}
\end{align}
Substituting the expression for $h(x)$ into $\mathbb{E}_{f\sim\mu}\left|\mathscr{G}(f-\mathscr{D}\circ\mathscr{E}(f))(x)h(s)\right|$, and using the fact that $q(x)$ is bounded on $[0,1]$ along with Corollary \ref{differ_norm}, we have the inequality
\begin{equation*}
    \mathbb{E}_{f\sim\mu}\left|\mathscr{G}(f-\mathscr{D}\circ\mathscr{E}(f))(x)h(s)\right|\le \frac{C}{m^2}.
\end{equation*}
     After substituting this equation into (\ref{noname}), the required inequality can be obtained via a straightforward integration calculation.
     \qed
\end{proof}
\subsection{Examples of Pairs $(F_i,u_i)$ in the Loss Function \eqref{loss}}
\label{appendix 7.4}
\begin{figure}[!tbh]
\centering
\begin{minipage}{0.4\textwidth}
\centering
\includegraphics[width=\textwidth]{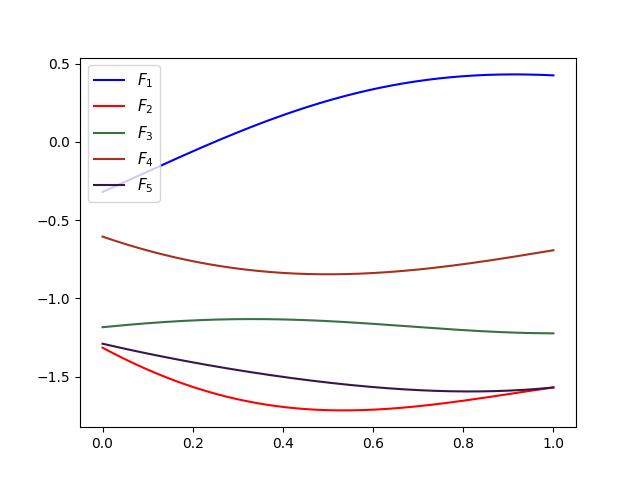}\\
\scriptsize{a)}
\end{minipage}
\begin{minipage}{0.4\textwidth}
\centering
\includegraphics[width=\textwidth]{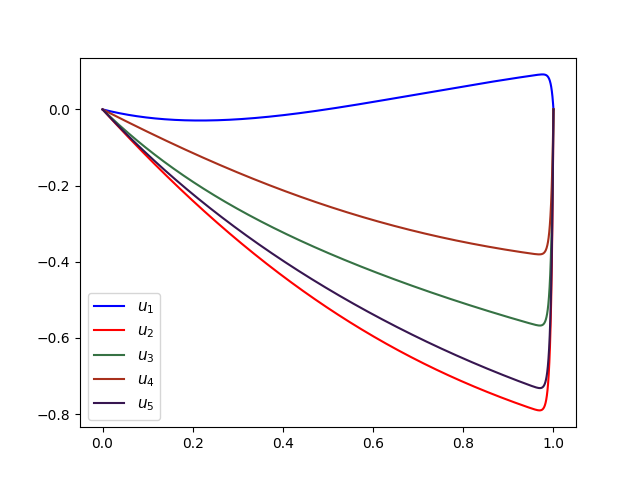}\\
\scriptsize{b)}
\end{minipage}
  \caption{a): Examples of random process $F_i\sim\mu$ with parameter $l=1$, generating by matrix decomposition method \cite{GRF_generation}.
b): Numerical solutions to Eq.~\eqref{equ} with $p(x)=x+1$, $q(x)=1$, $\varepsilon=0.01$, and right-hand side $f=F_i$ as presented in a).}
 \label{fig8}
\end{figure}
\subsection{Lipschitz Continuity of the Operator in Example \ref{exam_1}\label{appendix 7.5}}
\begin{lemma}
Let $\mathscr{G}: f\mapsto u$ denote the solution operator for the following boundary value problem: 
\begin{equation*}
    \left\{ \begin{array}{l}
-\varepsilon u''(x)+u'(x)=f(x),\quad x\in(0,1),\\
u(0)=u(1)=0,
\end{array} \right.
\end{equation*}
if $0<\varepsilon\le 1/2$, then the operator $\mathscr{G}$ is Lipschitz continuous.
\end{lemma}

\begin{proof}
As $\mathscr{G}$ is linear, it suffices to show that $\|u\|_{L^2}\lesssim\|f\|_{L^2}$. To achieve this, we substitute $v=e^{-x}u$ into (\ref{equal_equ}). Utilizing $uu'=\frac{1}{2}(u^2)'$ and $u(0)=u(1)=0$, we obtain:
\begin{equation*}
    \int_0^1\varepsilon (u')^2e^{-x}dx+\frac{1-\varepsilon}{2}\int_0^1e^{-x}u^2dx=\int_0^1e^{-x}ufdx.\label{7.4.1}
\end{equation*}
Next, we obtain the inequality
\begin{equation*}
    \frac{1-\varepsilon}{2e}\int_0^1u^2dx\le\|u\|_{L^2}\|f\|_{L^2}.
\end{equation*}
Notably, since $0<\varepsilon\ll1$, we can reasonably posit that $\varepsilon\le1/2$ and thereby complete the proof of the lemma.
\qed
\end{proof}

\end{document}